\documentclass[12pt,twoside,reqno]{amsart}

\usepackage{version}         % include, exclude and comment
 \usepackage{color}

\usepackage{tabulary}
\usepackage{array}

\usepackage{multirow}
\usepackage[OT1]{fontenc}
\usepackage{type1cm}
\usepackage{amssymb}
\usepackage{mathrsfs}
\usepackage[english]{babel}
\usepackage[latin1]{inputenc}
\usepackage[T1]{fontenc}
\usepackage{palatino}
\usepackage{amsfonts}
\usepackage{amsmath}
\usepackage{amssymb}
\usepackage{amsthm}
\usepackage{bbm}
\usepackage{extarrows}
\usepackage{graphicx}
\usepackage[usenames,dvipsnames]{xcolor}
\usepackage{wrapfig}
\usepackage{type1cm}
\usepackage[bf]{caption}
\usepackage{esint}
\usepackage{version}
\usepackage[colorlinks = true, citecolor = black, linkcolor = black, urlcolor = black, pdfstartview=FitH]{hyperref}
\usepackage{caption}
\usepackage{tikz}
\usepackage{bbding, wasysym}
\usepackage{enumitem}
\usepackage{float}

\usepackage[left=2.3cm,top=3cm,right=2.3cm]{geometry}
\numberwithin{equation}{section}

\theoremstyle{plain}
\newtheorem{theorem}{Theorem}[section]

\newtheorem{corollary}[theorem]{Corollary}
\newtheorem{proposition}[theorem]{Proposition}

\theoremstyle{remark}
\newtheorem{remark}[theorem]{Remark}

\newtheorem*{ack}{Acknowledgement}

\theoremstyle{definition}

\newcommand{\R}{\mathbb{R}}
\newcommand{\C}{\mathbb{C}}

\newcommand{\N}{\mathbb{N}}

\newcommand{\cH}{\mathcal{H}}

\newcommand{\cL}{\mathcal{L}}

\newcommand{\cA}{\mathcal{A}}

\begin{document}

\title{Generalized Koch curves and Thue-Morse sequences}

\author{Yao-Qiang Li}

\address{Institut de Math\'ematiques de Jussieu - Paris Rive Gauche \\
         Sorbonne Universit\'e - Campus Pierre et Marie Curie\\
         Paris, 75005 \\
         France}

\email{yaoqiang.li@imj-prg.fr\quad yaoqiang.li@etu.upmc.fr}

\address{School of Mathematics \\
         South China University of Technology \\
         Guangzhou, 510641 \\
         P.R. China}

\email{scutyaoqiangli@gmail.com\quad scutyaoqiangli@qq.com}


\begin{abstract}
Let $(t_n)_{n\ge0}$ be the well konwn $\pm1$ Thue-Morse sequence
$$+1,-1,-1,+1,-1,+1,+1,-1,\cdots.$$
Since the 1982-1983 work of Coquet and Dekking, it is known that $\sum_{k<n}t_ke^\frac{2k\pi i}{3}$ is strongly related to the famous Koch curve. As a natural generalization, for integer $m\ge1$, we use $\sum_{k<n}\delta_ke^\frac{2k\pi i}{m}$ to define generalized Koch curve, where $(\delta_n)_{n\ge0}$ is the generalized Thue-Morse sequence defined to be the unique fixed point of the morphism
$$+1\mapsto+1,+\delta_1,\cdots,+\delta_m$$
$$-1\mapsto-1,-\delta_1,\cdots,-\delta_m$$
beginning with $\delta_0=+1$ and $\delta_1,\cdots,\delta_m\in\{+1,-1\}$, and we prove that generalized Koch curves are the attractors of corresponding iterated function systems. For the case that $m\ge2$, $\delta_0=\cdots=\delta_{\lfloor\frac{m}{4}\rfloor}=+1$, $\delta_{\lfloor\frac{m}{4}\rfloor+1}=\cdots=\delta_{m-\lfloor\frac{m}{4}\rfloor-1}=-1$ and $\delta_{m-\lfloor\frac{m}{4}\rfloor}=\cdots=\delta_m=+1$, the open set condition holds, and then the corresponding generalized Koch curve has Hausdorff, packing and box dimension $\log(m+1)/\log|\sum_{k=0}^m\delta_ke^{\frac{2k\pi i}{m}}|$, where taking $m=3$ and then $\delta_0=+1,\delta_1=\delta_2=-1,\delta_3=+1$ will recover the result on the classical Koch curve.
\end{abstract}

\maketitle

\section{Introduction}

Let $\N$, $\N_0$, $\R$ and $\C$ be the sets of positive integers $1,2,3,\cdots$, non-negative integers $0,1,2,\cdots$, real numbers and complex numbers respectively. Denote the base of the natural logarithm by $e$ and the imaginary unit by $i$ as usual. Let $t=(t_n)_{n\ge0}$ be the classical $\pm1$ \textit{Thue-Morse sequence} (see \cite{AS99,AS03,T12})
$$+1,-1,-1,+1,-1,+1,+1,-1,\cdots.$$
It is well known that $t_n=(-1)^{s(n)}$ for all $n\in\N_0$ where $s(n)$ denotes the sum of binary digits of $n$. In the 1983 paper \cite{C83}, Coquet interested in the behavior of the sum $\sum_{k<n}(-1)^{s(3k)}$, introduced $\sum_{k<n}t_ke^\frac{2k\pi i}{3}$ and obtained the Koch curve \cite{V06} as a by-product in \cite[Page 111]{C83}. In addition, Dekking found in \cite[Pages 32-05 and 32-06]{D82b} that the points
$$p(0):=0,\quad p(n):=\sum_{k=0}^{n-1}t_ke^{\frac{2k\pi i}{3}}\quad (n=1,2,3,\cdots)$$
traverse the unscaled Koch curve on the complex plane (see also \cite[Page 107]{D91} and \cite[Page 304]{G17}). For more on the relation between the Koch curve and the Thue-Morse sequence, we refer the reader to \cite{AS07,MH05,Z16}.

For $m\in\N$ and $\delta_1,\cdots,\delta_m\in\{+1,-1\}$, we define the \textit{$(+1,\delta_1,\cdots,\delta_m)$-Thue-Morse sequence} $\delta=(\delta_n)_{n\ge0}$ to be the unique fixed point of the morphism
$$+1\mapsto+1,+\delta_1,\cdots,+\delta_m$$
$$-1\mapsto-1,-\delta_1,\cdots,-\delta_m$$
beginning with $\delta_0=+1$. Let
$$p_{m,\delta}(0):=0\quad\text{and}\quad p_{m,\delta}(n):=\sum_{k=0}^{n-1}\delta_ke^{\frac{2k\pi i}{m}}\quad\text{for }n=1,2,3,\cdots.$$
Noting that the classical $\pm1$ Thue-Morse sequence is not only the $(+1,-1)$ but also the $(+1,-1,-1,+1)$-Thue-Morse sequence in our terms, the above $p_{m,\delta}$ depends not only on $\delta$ but also on $m$. For $n\in\N_0$, let
$$P_{m,\delta}(n):=\bigcup_{k=1}^{(m+1)^n}[p_{m,\delta}(k-1),p_{m,\delta}(k)]$$
be the polygonal line connecting the points $p_{m,\delta}(0),p_{m,\delta}(1),\cdots,p_{m,\delta}((m+1)^n)$ one by one, where $[z_1,z_2]:=\{cz_1+(1-c)z_2:c\in[0,1]\}$ is the segment connecting $z_1$ and $z_2$ on the complex plane $\C$. In addition, if $p_{m,\delta}(m+1)\neq0$, for all $j\in\{0,1,\cdots,m\}$, we define $S_{m,\delta,j}:\C\to\C$ by
$$S_{m,\delta,j}(z):=\frac{p_{m,\delta}(j)+\delta_je^\frac{2j\pi i}{m}z}{p_{m,\delta}(m+1)}\quad\text{for }z\in\C.$$
When $|p_{m,\delta}(m+1)|>1$, obviously $S_{m,\delta,0},S_{m,\delta,1},\cdots,S_{m,\delta,m}$ are all contracting similarities, and we call $\{S_{m,\delta,j}\}_{0\le j\le m}$ the \textit{$(+1,\delta_1,\cdots,\delta_m)$-IFS} (\textit{iterated function system}). We can see that the attractor of the $(+1,-1,-1,+1)$-IFS is exactly the Koch curve.

For simplification, if $m$ and the $(+1,\delta_1,\cdots,\delta_m)$-Thue-Morse sequence $\delta$ are understood from the context, we use $p$, $P$ and $S_j$ instead of $p_{m,\delta}$, $P_{m,\delta}$ and $S_{m,\delta,j}$ respectively.

Let $d_H$ be the Hausdorff metric and write $cZ:=\{cz:z\in Z\}$ for any $c\in\C$ and $Z\subset\C$. The following is our main result.

\begin{theorem}\label{main} Let $m\in\N$, $\delta_0=+1$, $\delta_1,\cdots,\delta_m\in\{+1,-1\}$ and $\delta=(\delta_n)_{n\ge0}$ be the $(+1,\delta_1,\cdots,\delta_m)$-Thue-Morse sequence. If $|p(m+1)|>1$, then there exists a unique compact set $K\subset\C$ such that
$$(p(m+1))^{-n}P(n)\overset{d_H}{\longrightarrow}K\quad\text{as }n\to\infty,$$
and $K$ is a continuous image of $[0,1]$. Moreover, $K$ is the unique attractor of the $(+1,\delta_1,\cdots,\delta_m)$-IFS $\{S_j\}_{0\le j\le m}$. That is, $K$ is the unique non-empty compact set such that
$$K=\bigcup_{j=0}^m S_j(K).$$
Furthermore,
$$\dim_HK=\frac{\log(m+1)}{\log|p(m+1)|}$$
if and only if there exists $\varepsilon>0$ such that
$$\varliminf_{n\to\infty}\frac{\cL((P(n))^\varepsilon)}{(m+1)^n}>0,$$
where $\cL$ is the Lebesgue measure on the plane and $A^\varepsilon:=\{z\in\C:|z-a|<\varepsilon\text{ for some }a\in A\}$ for $A\subset\C$.
\end{theorem}

We call $K$ in Theorem \ref{main} the \textit{$(+1,\delta_1,\cdots,\delta_m)$-Koch curve}. See the following figures for some examples for $m=3$ and $4$. Note that the classical Koch curve is exactly the $(+1,-1,-1,+1)$-Koch curve in our terms.

\newpage

\newcolumntype{F}{>{$}c<{$}}
$$
\begin{tabular}{FF}
\begin{tikzpicture}[scale = 5.5/3]
\draw node[above]at(0,0){$0$};
\draw node[above]at(3,0){$1$};
\draw[dashed,scale=3](0,0)--(1,0);
\draw[-,line width=0.7](0,0)--(1,0)--(1.5,0-0.866)--(2,0)--(3,0);
\end{tikzpicture}
&\begin{tikzpicture}[scale = 5.5/2.645751311064591,rotate=-40.89339465]
\draw node[above]at(0,0){$0$};
\draw node[below]at(2,1.732){$1$};
\draw[dashed,scale=2.645751311064591,rotate=40.89339465](0,0)--(1,0);
\draw[-,line width=0.7](0,0)--(1,0)--(0.5,0.866)--(1,1.732)--(2,1.732);
\end{tikzpicture}
\\
\begin{tikzpicture}[scale = 5.5/9]
\draw node[above]at(0,0){$0$};
\draw node[above]at(9,0){$1$};
\draw[dashed,scale=9/3](0,0)--(1,0)--(1.5,0-0.866)--(2,0)--(3,0);
\draw[-,line width=0.7](0,0)--(1,0)--(1.5,0-0.866)--(2,0)--(3,0)--(3.5,0-0.866)--(3,0-0.866*2)--(4,0-0.866*2)--(4.5,0-0.866*3);
\draw[-,line width=0.7](9-0,0)--(9-1,0)--(9-1.5,0-0.866)--(9-2,0)--(9-3,0)--(9-3.5,0-0.866)--(9-3,0-0.866*2)--(9-4,0-0.866*2)--(9-4.5,0-0.866*3);
\end{tikzpicture}
&
\begin{tikzpicture}[scale = 5.5/7,rotate=-81.786789298262]
\draw node[above]at(0,0){$0$};
\draw node[below]at(1,8*0.866){$1$};
\draw[dashed,scale=7/2.645751311064591,rotate=81.786789298262-40.89339465](0,0)--(1,0)--(0.5,0.866)--(1,1.732)--(2,1.732);
\draw[-,line width=0.7](0,0)--(1,0)--(0.5,0.866)--(1,1.732)--(2,1.732)--(1.5,0.866*3)--(1,1.732)--(0,1.732)--(-0.5,3*0.866)--(0,4*0.866) --(-1,4*0.866)--(-1.5,5*0.866)--(-1,6*0.866)--(0,6*0.866)--(-0.5,7*0.866)--(0,8*0.866)--(1,8*0.866);
\end{tikzpicture}
\\
\begin{tikzpicture}[scale = 5.5/27]
\draw node[above]at(0,0){$0$};
\draw node[above]at(27,0){$1$};
\draw[dashed,scale=27/9](0,0)--(1,0)--(1.5,0-0.866)--(2,0)--(3,0)--(3.5,0-0.866)--(3,0-0.866*2)--(4,0-0.866*2)--(4.5,0-0.866*3);
\draw[dashed,scale=27/9](9-0,0)--(9-1,0)--(9-1.5,0-0.866)--(9-2,0)--(9-3,0)--(9-3.5,0-0.866)--(9-3,0-0.866*2)--(9-4,0-0.866*2)--(9-4.5,0-0.866*3);
\draw[-,line width=0.7](0,0)--(1,0)--(1.5,0-0.866)--(2,0)--(3,0)--(3.5,0-0.866)--(3,0-0.866*2)--(4,0-0.866*2)--(4.5,0-0.866*3);
\draw[-,line width=0.7](9-0,0)--(9-1,0)--(9-1.5,0-0.866)--(9-2,0)--(9-3,0)--(9-3.5,0-0.866)--(9-3,0-0.866*2)--(9-4,0-0.866*2)--(9-4.5,0-0.866*3);
\draw[-,line width=0.7,xshift=9cm,rotate=-60](0,0)--(1,0)--(1.5,0-0.866)--(2,0)--(3,0)--(3.5,0-0.866)--(3,0-0.866*2)--(4,0-0.866*2)--(4.5,0-0.866*3);
\draw[-,line width=0.7,xshift=9cm,rotate=-60](9-0,0)--(9-1,0)--(9-1.5,0-0.866)--(9-2,0)--(9-3,0)--(9-3.5,0-0.866)--(9-3,0-0.866*2)--(9-4,0-0.866*2)--(9-4.5,0-0.866*3);
\draw[-,line width=0.7,xshift=13.5cm,yshift=-7.794228634059948cm,rotate=60](0,0)--(1,0)--(1.5,0-0.866)--(2,0)--(3,0)--(3.5,0-0.866)--(3,0-0.866*2)--(4,0-0.866*2)--(4.5,0-0.866*3);
\draw[-,line width=0.7,xshift=13.5cm,yshift=-7.794228634059948cm,rotate=60](9-0,0)--(9-1,0)--(9-1.5,0-0.866)--(9-2,0)--(9-3,0)--(9-3.5,0-0.866)--(9-3,0-0.866*2)--(9-4,0-0.866*2)--(9-4.5,0-0.866*3);
\draw[-,line width=0.7,xshift=18cm](0,0)--(1,0)--(1.5,0-0.866)--(2,0)--(3,0)--(3.5,0-0.866)--(3,0-0.866*2)--(4,0-0.866*2)--(4.5,0-0.866*3);
\draw[-,line width=0.7,xshift=18cm](9-0,0)--(9-1,0)--(9-1.5,0-0.866)--(9-2,0)--(9-3,0)--(9-3.5,0-0.866)--(9-3,0-0.866*2)--(9-4,0-0.866*2)--(9-4.5,0-0.866*3);
\end{tikzpicture}&
\begin{tikzpicture}[scale = 5.5/18.52025917745213,rotate=-122.6801839473927]
\draw node[above]at(0,0){$0$};
\draw node[below]at(-10,18*0.866){$1$};
\draw[dashed,scale=18.52025917745213/7,rotate=122.6801839473927-81.786789298262](0,0)--(1,0)--(0.5,0.866)--(1,1.732)--(2,1.732)--(1.5,0.866*3)--(1,1.732)--(0,1.732)--(-0.5,3*0.866)--(0,4*0.866) --(-1,4*0.866)--(-1.5,5*0.866)--(-1,6*0.866)--(0,6*0.866)--(-0.5,7*0.866)--(0,8*0.866)--(1,8*0.866);
\draw[-,line width=0.7](0,0)--(1,0)--(0.5,0.866)--(1,1.732)--(2,1.732)--(1.5,0.866*3)--(1,1.732)--(0,1.732)--(-0.5,3*0.866)--(0,4*0.866) --(-1,4*0.866)--(-1.5,5*0.866)--(-1,6*0.866)--(0,6*0.866)--(-0.5,7*0.866)--(0,8*0.866)--(1,8*0.866)--(0.5,9*0.866) --(0,8*0.866)--(-1,8*0.866)--(-1.5,9*0.866)--(-2,8*0.866)--(-1,8*0.866)--(-0.5,7*0.866)--(-1,6*0.866)--(-2,6*0.866)--(-1.5,5*0.866)--(-2,4*0.866)--(-3,4*0.866)--(-3.5,5*0.866) --(-4,4*0.866)--(-5,4*0.866)--(-5.5,5*0.866)--(-5,6*0.866)--(-6,6*0.866)--(-6.5,7*0.866)--(-6,8*0.866)--(-7,8*0.866)--(-6.5,7*0.866)--(-7,6*0.866)--(-8,6*0.866)--(-8.5,7*0.866)--(-9,6*0.866) --(-10,6*0.866)--(-10.5,7*0.866)--(-10,8*0.866)--(-11,8*0.866)--(-11.5,9*0.866)--(-11,10*0.866)--(-10,10*0.866)--(-10.5,11*0.866)--(-10,12*0.866)--(-9,12*0.866)--(-9.5,13*0.866) --(-10,12*0.866)--(-11,12*0.866)--(-11.5,13*0.866)--(-11,14*0.866)--(-12,14*0.866)--(-12.5,15*0.866)--(-12,16*0.866)--(-11,16*0.866)--(-11.5,17*0.866)--(-11,18*0.866)--(-10,18*0.866);
\end{tikzpicture}
\\
\vdots
&
\vdots
\\
&
\\
\quad\quad\text{\begin{tiny}Construction of the $(+1,-1,-1,+1)$-Koch curve.\end{tiny}}\quad\quad
&
\quad\quad\text{\begin{tiny}Construction of the $(+1,+1,-1,+1)$-Koch curve.\end{tiny}}\quad\quad
\\
&
\\
&
\\
\begin{tikzpicture}[scale = 2.25/1.732]
\draw node[above]at(0,0){$0$};
\draw node[above]at(1.732,0){$1$};
\draw[dashed,scale=1.732](0,0)--(1,0);
\draw[-,line width=0.7](0,0)--(0,-1)--(0.866,-0.5)--(0.866*2,-1)--(0.866*2,0);
\end{tikzpicture}
&
\begin{tikzpicture}[scale = 2.25/1.732]
\draw node[below]at(0,0){$0$};
\draw node[below]at(1.732,0){$1$};
\draw[dashed,scale=1.732](0,0)--(1,0);
\draw[-,line width=0.7](0,0)--(0,1)--(0.866,1.5)--(1.732,1)--(1.732,0);
\end{tikzpicture}
\\
\begin{tikzpicture}[scale = 2.25/3]
\draw node[above]at(0,0){$0$};
\draw node[above]at(3,0){$1$};
\draw[dashed,scale=3/1.732](0,0)--(0,-1)--(0.866,-0.5)--(0.866*2,-1)--(0.866*2,0);
\draw[-,line width=0.7](0,0)--(-1,0)--(-0.5,-0.866)--(-1,-0.866*2)--(0,-0.866*2)--(0.5,-0.866*3)--(1.5,-0.866);
\draw[-,line width=0.7](1,-0.866*2)--(2,-0.866*2);
\draw[-,line width=0.7](3-0,0)--(3+1,0)--(3+0.5,-0.866)--(3+1,-0.866*2)--(3,-0.866*2)--(3-0.5,-0.866*3)--(1.5,-0.866);
\end{tikzpicture}
&
\begin{tikzpicture}[scale = 2.25/3]
\draw node[below]at(0,0){$0$};
\draw node[below]at(3,0){$1$};
\draw[dashed,scale=3/1.732](0,0)--(0,1)--(0.866,1.5)--(1.732,1)--(1.732,0);
\draw[-,line width=0.7](0,0)--(-1,0)--(-1.5,0.866)--(-1,1.732)--(0,1.732)--(-0.5,2.598)--(0,3.464)--(1,3.464)--(1.5,2.598);
\draw[-,line width=0.7](3,0)--(4,0)--(4.5,0.866)--(4,1.732)--(3,1.732)--(3.5,2.598)--(3,3.464)--(2,3.464)--(1.5,2.598);
\end{tikzpicture}
\\
\begin{tikzpicture}[scale = 2.25/5.196]
\draw node[right]at(0,0){$0$};
\draw node[left]at(5.196,0){$1$};
\draw[dashed,scale=5.196/3](0,0)--(-1,0)--(-0.5,-0.866)--(-1,-0.866*2)--(0,-0.866*2)--(0.5,-0.866*3)--(1.5,-0.866);
\draw[dashed,scale=5.196/3](1,-0.866*2)--(2,-0.866*2);
\draw[dashed,scale=5.196/3](3-0,0)--(3+1,0)--(3+0.5,-0.866)--(3+1,-0.866*2)--(3,-0.866*2)--(3-0.5,-0.866*3)--(1.5,-0.866);
\draw[-,line width=0.7,rotate=-90](0,0)--(-1,0)--(-0.5,-0.866)--(-1,-0.866*2)--(0,-0.866*2)--(0.5,-0.866*3)--(1.5,-0.866);
\draw[-,line width=0.7,rotate=-90](1,-0.866*2)--(2,-0.866*2);
\draw[-,line width=0.7,rotate=-90](3-0,0)--(3+1,0)--(3+0.5,-0.866)--(3+1,-0.866*2)--(3,-0.866*2)--(3-0.5,-0.866*3)--(1.5,-0.866);
\draw[-,line width=0.7,yshift=-3cm,rotate=30](0,0)--(-1,0)--(-0.5,-0.866)--(-1,-0.866*2)--(0,-0.866*2)--(0.5,-0.866*3)--(1.5,-0.866);
\draw[-,line width=0.7,yshift=-3cm,rotate=30](1,-0.866*2)--(2,-0.866*2);
\draw[-,line width=0.7,yshift=-3cm,rotate=30](3-0,0)--(3+1,0)--(3+0.5,-0.866)--(3+1,-0.866*2)--(3,-0.866*2)--(3-0.5,-0.866*3)--(1.5,-0.866);
\draw[-,line width=0.7,xshift=2.598cm,yshift=-1.5cm,rotate=-30](0,0)--(-1,0)--(-0.5,-0.866)--(-1,-0.866*2)--(0,-0.866*2)--(0.5,-0.866*3)--(1.5,-0.866);
\draw[-,line width=0.7,xshift=2.598cm,yshift=-1.5cm,rotate=-30](1,-0.866*2)--(2,-0.866*2);
\draw[-,line width=0.7,xshift=2.598cm,yshift=-1.5cm,rotate=-30](3-0,0)--(3+1,0)--(3+0.5,-0.866)--(3+1,-0.866*2)--(3,-0.866*2)--(3-0.5,-0.866*3)--(1.5,-0.866);
\draw[-,line width=0.7,xshift=5.196cm,yshift=-3cm,rotate=90](0,0)--(-1,0)--(-0.5,-0.866)--(-1,-0.866*2)--(0,-0.866*2)--(0.5,-0.866*3)--(1.5,-0.866);
\draw[-,line width=0.7,xshift=5.196cm,yshift=-3cm,rotate=90](1,-0.866*2)--(2,-0.866*2);
\draw[-,line width=0.7,xshift=5.196cm,yshift=-3cm,rotate=90](3-0,0)--(3+1,0)--(3+0.5,-0.866)--(3+1,-0.866*2)--(3,-0.866*2)--(3-0.5,-0.866*3)--(1.5,-0.866);
\end{tikzpicture}
&
\begin{tikzpicture}[scale = 2.25/5.196]
\draw node[right]at(0,0){$0$};
\draw node[left]at(5.196,0){$1$};
\draw[dashed,scale=5.196/3](0,0)--(-1,0)--(-1.5,0.866)--(-1,1.732)--(0,1.732)--(-0.5,2.598)--(0,3.464)--(1,3.464)--(1.5,2.598);
\draw[dashed,scale=5.196/3](3,0)--(4,0)--(4.5,0.866)--(4,1.732)--(3,1.732)--(3.5,2.598)--(3,3.464)--(2,3.464)--(1.5,2.598);
\draw[-,line width=0.7,rotate=90](0,0)--(-1,0)--(-1.5,0.866)--(-1,1.732)--(0,1.732)--(-0.5,2.598)--(0,3.464)--(1,3.464)--(1.5,2.598);
\draw[-,line width=0.7,rotate=90](3,0)--(4,0)--(4.5,0.866)--(4,1.732)--(3,1.732)--(3.5,2.598)--(3,3.464)--(2,3.464)--(1.5,2.598);
\draw[-,line width=0.7,yshift=3cm,rotate=30](0,0)--(-1,0)--(-1.5,0.866)--(-1,1.732)--(0,1.732)--(-0.5,2.598)--(0,3.464)--(1,3.464)--(1.5,2.598);
\draw[-,line width=0.7,yshift=3cm,rotate=30](3,0)--(4,0)--(4.5,0.866)--(4,1.732)--(3,1.732)--(3.5,2.598)--(3,3.464)--(2,3.464)--(1.5,2.598);
\draw[-,line width=0.7,xshift=2.598cm,yshift=4.5cm,rotate=-30](0,0)--(-1,0)--(-1.5,0.866)--(-1,1.732)--(0,1.732)--(-0.5,2.598)--(0,3.464)--(1,3.464)--(1.5,2.598);
\draw[-,line width=0.7,xshift=2.598cm,yshift=4.5cm,rotate=-30](3,0)--(4,0)--(4.5,0.866)--(4,1.732)--(3,1.732)--(3.5,2.598)--(3,3.464)--(2,3.464)--(1.5,2.598);
\draw[-,line width=0.7,xshift=5.196cm,yshift=3cm,rotate=-90](0,0)--(-1,0)--(-1.5,0.866)--(-1,1.732)--(0,1.732)--(-0.5,2.598)--(0,3.464)--(1,3.464)--(1.5,2.598);
\draw[-,line width=0.7,xshift=5.196cm,yshift=3cm,rotate=-90](3,0)--(4,0)--(4.5,0.866)--(4,1.732)--(3,1.732)--(3.5,2.598)--(3,3.464)--(2,3.464)--(1.5,2.598);
\end{tikzpicture}
\\
\vdots
&
\vdots
\\
&
\\
\quad\quad\text{\begin{tiny}Construction of the $(+1,+1,-1,-1)$-Koch curve.\end{tiny}}\quad\quad
&
\quad\quad\text{\begin{tiny}Construction of the $(+1,-1,+1,-1)$-Koch curve.\end{tiny}}\quad\quad
\end{tabular}
$$

\newcolumntype{F}{>{$}c<{$}}
$$
\begin{tabular}{FF}
\begin{tikzpicture}[scale = 6/3]
\draw node[below]at(0,0){$0$};
\draw node[below]at(3,0){$1$};
\draw[dashed,scale=3](0,0)--(1,0);
\draw[-,line width=0.7](0,0)--(1,0)--(1,1)--(2,1)--(2,0)--(3,0);
\end{tikzpicture}
&
\begin{tikzpicture}[scale = 6/3.60555127546399,rotate=-33.69006752597977083497]
\draw node[above]at(0,0){$0$};
\draw node[below]at(3,2){$1$};
\draw[dashed,scale=3.60555127546399,rotate=33.69006752597977083497](0,0)--(1,0);
\draw[-,line width=0.7](0,0)--(1,0)--(1,1)--(2,1)--(2,2)--(3,2);
\end{tikzpicture}
\\
\begin{tikzpicture}[scale = 6/9]
\draw node[below]at(0,0){$0$};
\draw node[below]at(9,0){$1$};
\draw[dashed,scale=9/3](0,0)--(1,0)--(1,1)--(2,1)--(2,0)--(3,0);
\draw[-,line width=0.7](0,0)--(1,0)--(1,1)--(2,1)--(2,0)--(3,0)--(3,1)--(2,1)--(2,2)--(3,2)--(3,3)--(4,3)--(4,4)--(5,4)--(5,3)--(6,3)--(6,2)--(7,2)--(7,1)--(6,1)--(6,0)--(7,0)--(7,1)--(8,1)--(8,0)--(9,0);
\end{tikzpicture}
&
\begin{tikzpicture}[scale = 6/13,rotate=-2*33.69006752597977083497]
\draw node[above]at(0,0){$0$};
\draw node[below]at(5,12){$1$};
\draw[dashed,scale=13/3.60555127546399,rotate=+33.69006752597977083497](0,0)--(1,0)--(1,1)--(2,1)--(2,2)--(3,2);
\draw[-,line width=0.7](0,0)--(1,0)--(1,1)--(2,1)--(2,2)--(3,2)--(3,3)--(2,3)--(2,4)--(1,4)--(1,5)--(2,5)--(2,6)--(3,6)--(3,7)--(4,7)--(4,8)--(3,8)--(3,9)--(2,9)--(2,10)--(3,10)--(3,11)--(4,11)--(4,12)--(5,12);
\end{tikzpicture}
\\
\begin{tikzpicture}[scale = 6/27]
\draw node[below]at(0,0){$0$};
\draw node[below]at(27,0){$1$};
\draw[dashed,scale=27/9](0,0)--(1,0)--(1,1)--(2,1)--(2,0)--(3,0)--(3,1)--(2,1)--(2,2)--(3,2)--(3,3)--(4,3)--(4,4)--(5,4)--(5,3)--(6,3)--(6,2)--(7,2)--(7,1)--(6,1)--(6,0)--(7,0)--(7,1)--(8,1)--(8,0)--(9,0);
\draw[-,line width=0.7](0,0)--(1,0)--(1,1)--(2,1)--(2,0)--(3,0)--(3,1)--(2,1)--(2,2)--(3,2)--(3,3)--(4,3)--(4,4)--(5,4)--(5,3)--(6,3)--(6,2)--(7,2)--(7,1)--(6,1)--(6,0)--(7,0)--(7,1)--(8,1)--(8,0)--(9,0);
\draw[-,line width=0.7,xshift=9cm,rotate=90](0,0)--(1,0)--(1,1)--(2,1)--(2,0)--(3,0)--(3,1)--(2,1)--(2,2)--(3,2)--(3,3)--(4,3)--(4,4)--(5,4)--(5,3)--(6,3)--(6,2)--(7,2)--(7,1)--(6,1)--(6,0)--(7,0)--(7,1)--(8,1)--(8,0)--(9,0);
\draw[-,line width=0.7,xshift=9cm,yshift=9cm](0,0)--(1,0)--(1,1)--(2,1)--(2,0)--(3,0)--(3,1)--(2,1)--(2,2)--(3,2)--(3,3)--(4,3)--(4,4)--(5,4)--(5,3)--(6,3)--(6,2)--(7,2)--(7,1)--(6,1)--(6,0)--(7,0)--(7,1)--(8,1)--(8,0)--(9,0);
\draw[-,line width=0.7,xshift=18cm,yshift=9cm,rotate=-90](0,0)--(1,0)--(1,1)--(2,1)--(2,0)--(3,0)--(3,1)--(2,1)--(2,2)--(3,2)--(3,3)--(4,3)--(4,4)--(5,4)--(5,3)--(6,3)--(6,2)--(7,2)--(7,1)--(6,1)--(6,0)--(7,0)--(7,1)--(8,1)--(8,0)--(9,0);
\draw[-,line width=0.7,xshift=18cm](0,0)--(1,0)--(1,1)--(2,1)--(2,0)--(3,0)--(3,1)--(2,1)--(2,2)--(3,2)--(3,3)--(4,3)--(4,4)--(5,4)--(5,3)--(6,3)--(6,2)--(7,2)--(7,1)--(6,1)--(6,0)--(7,0)--(7,1)--(8,1)--(8,0)--(9,0);
\end{tikzpicture}
&
\begin{tikzpicture}[scale =6/46.87216658103186,rotate=-33.69006752597977083497]
\draw node[above]at(0,0){$0$};
\draw node[below]at(39,26){$1$};
\draw[dashed,scale=46.87216658103186/13,rotate=-33.69006752597977083497](0,0)--(1,0)--(1,1)--(2,1)--(2,2)--(3,2)--(3,3)--(2,3)--(2,4)--(1,4)--(1,5)--(2,5)--(2,6)--(3,6)--(3,7)--(4,7)--(4,8)--(3,8)--(3,9)--(2,9)--(2,10)--(3,10)--(3,11)--(4,11)--(4,12)--(5,12);
\draw[-,line width=0.7,rotate=-2*33.69006752597977083497](0,0)--(1,0)--(1,1)--(2,1)--(2,2)--(3,2)--(3,3)--(2,3)--(2,4)--(1,4)--(1,5)--(2,5)--(2,6)--(3,6)--(3,7)--(4,7)--(4,8)--(3,8)--(3,9)--(2,9)--(2,10)--(3,10)--(3,11)--(4,11)--(4,12)--(5,12);
\draw[-,line width=0.7,xshift=13cm,rotate=90-2*33.69006752597977083497](0,0)--(1,0)--(1,1)--(2,1)--(2,2)--(3,2)--(3,3)--(2,3)--(2,4)--(1,4)--(1,5)--(2,5)--(2,6)--(3,6)--(3,7)--(4,7)--(4,8)--(3,8)--(3,9)--(2,9)--(2,10)--(3,10)--(3,11)--(4,11)--(4,12)--(5,12);
\draw[-,line width=0.7,xshift=13cm,yshift=13cm,rotate=-2*33.69006752597977083497](0,0)--(1,0)--(1,1)--(2,1)--(2,2)--(3,2)--(3,3)--(2,3)--(2,4)--(1,4)--(1,5)--(2,5)--(2,6)--(3,6)--(3,7)--(4,7)--(4,8)--(3,8)--(3,9)--(2,9)--(2,10)--(3,10)--(3,11)--(4,11)--(4,12)--(5,12);
\draw[-,line width=0.7,xshift=26cm,yshift=13cm,rotate=90-2*33.69006752597977083497](0,0)--(1,0)--(1,1)--(2,1)--(2,2)--(3,2)--(3,3)--(2,3)--(2,4)--(1,4)--(1,5)--(2,5)--(2,6)--(3,6)--(3,7)--(4,7)--(4,8)--(3,8)--(3,9)--(2,9)--(2,10)--(3,10)--(3,11)--(4,11)--(4,12)--(5,12);
\draw[-,line width=0.7,xshift=26cm,yshift=26cm,rotate=-2*33.69006752597977083497](0,0)--(1,0)--(1,1)--(2,1)--(2,2)--(3,2)--(3,3)--(2,3)--(2,4)--(1,4)--(1,5)--(2,5)--(2,6)--(3,6)--(3,7)--(4,7)--(4,8)--(3,8)--(3,9)--(2,9)--(2,10)--(3,10)--(3,11)--(4,11)--(4,12)--(5,12);
\end{tikzpicture}
\\
\vdots
&
\vdots
\\
&
\\
\quad\text{\begin{tiny}Construction of the $(+1,+1,-1,+1,+1)$-Koch curve.\end{tiny}}\quad
&
\quad\text{\begin{tiny}Construction of the $(+1,+1,-1,-1,+1)$-Koch curve.\end{tiny}}\quad
\\
&
\\
&
\\
\begin{tikzpicture}[scale = 3/2.2360679774998,rotate=-63.43494882292198155938]
\draw node[above]at(0,0){$0$};
\draw node[below]at(1,2){$1$};
\draw[dashed,scale = 2.2360679774998,rotate=63.43494882292198155938](0,0)--(1,0);
\draw[-,line width=0.7](0,0)--(1,0)--(1,1)--(0,1)--(0,2)--(1,2);
\end{tikzpicture}
&
\begin{tikzpicture}[scale = 3/2.2360679774998,rotate=-63.43494882292198155938]
\draw node[above]at(0,0){$0$};
\draw node[above]at(1,2){$1$};
\draw[dashed,scale = 2.2360679774998,rotate=63.43494882292198155938](0,0)--(1,0);
\draw[-,line width=0.7](0,0)--(1,0)--(1,1)--(2,1)--(2,2)--(1,2);
\end{tikzpicture}
\\
\begin{tikzpicture}[scale =3/5,rotate=-2*63.43494882292198155938]
\draw node[above]at(0,0){$0$};
\draw node[below]at(-3,4){$1$};
\draw[dashed,scale = 5/2.2360679774998,rotate=63.43494882292198155938](0,0)--(1,0)--(1,1)--(0,1)--(0,2)--(1,2);
\draw[-,line width=0.7](0,0)--(1,0)--(1,1)--(0,1)--(0,2)--(1,2)--(1,3)--(0,3)--(0,2)--(-1,2)--(-1,3)--(-2,3)--(-2,2)--(-1,2)--(-1,1)--(-2,1)--(-2,2)--(-3,2)--(-3,1)--(-4,1)--(-4,2)--(-3,2)--(-3,3)--(-4,3)--(-4,4)--(-3,4);
\end{tikzpicture}
&
\begin{tikzpicture}[scale = 3/5,rotate=-2*63.43494882292198155938]
\draw node[above]at(0,0){$0$};
\draw node[left]at(-3,4){$1$};
\draw[dashed,scale = 5/2.2360679774998,rotate=63.43494882292198155938](0,0)--(1,0)--(1,1)--(2,1)--(2,2)--(1,2);
\draw[-,line width=0.7](0,0)--(1,0)--(1,1)--(2,1)--(2,2)--(1,2)--(1,3)--(0,3)--(0,4)--(-1,4)--(-1,3)--(0,3)--(0,4)--(1,4)--(1,5)--(0,5)--(0,6)--(-1,6)--(-1,7)--(-2,7)--(-2,6)--(-3,6)--(-3,5)--(-4,5)--(-4,4)--(-3,4);
\end{tikzpicture}
\\
\begin{tikzpicture}[scale =3/11.18033988749895,rotate=-63.43494882292198155938]
\draw node[above]at(0,0){$0$};
\draw node[below]at(5,10){$1$};
\draw[dashed,scale=11.18033988749895/5,rotate=-63.43494882292198155938](0,0)--(1,0)--(1,1)--(0,1)--(0,2)--(1,2)--(1,3)--(0,3)--(0,2)--(-1,2)--(-1,3)--(-2,3)--(-2,2)--(-1,2)--(-1,1)--(-2,1)--(-2,2)--(-3,2)--(-3,1)--(-4,1)--(-4,2)--(-3,2)--(-3,3)--(-4,3)--(-4,4)--(-3,4);
\draw[-,line width=0.7,rotate=-2*63.43494882292198155938](0,0)--(1,0)--(1,1)--(0,1)--(0,2)--(1,2)--(1,3)--(0,3)--(0,2)--(-1,2)--(-1,3)--(-2,3)--(-2,2)--(-1,2)--(-1,1)--(-2,1)--(-2,2)--(-3,2)--(-3,1)--(-4,1)--(-4,2)--(-3,2)--(-3,3)--(-4,3)--(-4,4)--(-3,4);
\draw[-,line width=0.7,xshift=5cm,rotate=90-2*63.43494882292198155938](0,0)--(1,0)--(1,1)--(0,1)--(0,2)--(1,2)--(1,3)--(0,3)--(0,2)--(-1,2)--(-1,3)--(-2,3)--(-2,2)--(-1,2)--(-1,1)--(-2,1)--(-2,2)--(-3,2)--(-3,1)--(-4,1)--(-4,2)--(-3,2)--(-3,3)--(-4,3)--(-4,4)--(-3,4);
\draw[-,line width=0.7,xshift=5cm,yshift=5cm,rotate=180-2*63.43494882292198155938](0,0)--(1,0)--(1,1)--(0,1)--(0,2)--(1,2)--(1,3)--(0,3)--(0,2)--(-1,2)--(-1,3)--(-2,3)--(-2,2)--(-1,2)--(-1,1)--(-2,1)--(-2,2)--(-3,2)--(-3,1)--(-4,1)--(-4,2)--(-3,2)--(-3,3)--(-4,3)--(-4,4)--(-3,4);
\draw[-,line width=0.7,xshift=0cm,yshift=5cm,rotate=90-2*63.43494882292198155938](0,0)--(1,0)--(1,1)--(0,1)--(0,2)--(1,2)--(1,3)--(0,3)--(0,2)--(-1,2)--(-1,3)--(-2,3)--(-2,2)--(-1,2)--(-1,1)--(-2,1)--(-2,2)--(-3,2)--(-3,1)--(-4,1)--(-4,2)--(-3,2)--(-3,3)--(-4,3)--(-4,4)--(-3,4);
\draw[-,line width=0.7,xshift=0cm,yshift=10cm,rotate=-2*63.43494882292198155938](0,0)--(1,0)--(1,1)--(0,1)--(0,2)--(1,2)--(1,3)--(0,3)--(0,2)--(-1,2)--(-1,3)--(-2,3)--(-2,2)--(-1,2)--(-1,1)--(-2,1)--(-2,2)--(-3,2)--(-3,1)--(-4,1)--(-4,2)--(-3,2)--(-3,3)--(-4,3)--(-4,4)--(-3,4);
\end{tikzpicture}
&
\begin{tikzpicture}[scale = 3/11.18033988749895,rotate=-63.43494882292198155938]
\draw node[right]at(0,0){$0$};
\draw node[below]at(5,10){$1$};
\draw[dashed,scale=11.18033988749895/5,rotate=-63.43494882292198155938](0,0)--(1,0)--(1,1)--(2,1)--(2,2)--(1,2)--(1,3)--(0,3)--(0,4)--(-1,4)--(-1,3)--(0,3);
\draw[dashed,scale=11.18033988749895/5,rotate=-63.43494882292198155938](0,4)--(1,4)--(1,5)--(0,5)--(0,6)--(-1,6)--(-1,7)--(-2,7)--(-2,6)--(-3,6)--(-3,5)--(-4,5)--(-4,4)--(-3,4);
\draw[-,line width=0.7,rotate=-2*63.43494882292198155938](0,0)--(1,0)--(1,1)--(2,1)--(2,2)--(1,2)--(1,3)--(0,3)--(0,4)--(-1,4)--(-1,3)--(0,3)--(0,4)--(1,4)--(1,5)--(0,5)--(0,6)--(-1,6)--(-1,7)--(-2,7)--(-2,6)--(-3,6)--(-3,5)--(-4,5)--(-4,4)--(-3,4);
\draw[-,line width=0.7,xshift=5cm,rotate=90-2*63.43494882292198155938](0,0)--(1,0)--(1,1)--(2,1)--(2,2)--(1,2)--(1,3)--(0,3)--(0,4)--(-1,4)--(-1,3)--(0,3)--(0,4)--(1,4)--(1,5)--(0,5)--(0,6)--(-1,6)--(-1,7)--(-2,7)--(-2,6)--(-3,6)--(-3,5)--(-4,5)--(-4,4)--(-3,4);
\draw[-,line width=0.7,xshift=5cm,yshift=5cm,rotate=-2*63.43494882292198155938](0,0)--(1,0)--(1,1)--(2,1)--(2,2)--(1,2)--(1,3)--(0,3)--(0,4)--(-1,4)--(-1,3)--(0,3)--(0,4)--(1,4)--(1,5)--(0,5)--(0,6)--(-1,6)--(-1,7)--(-2,7)--(-2,6)--(-3,6)--(-3,5)--(-4,5)--(-4,4)--(-3,4);
\draw[-,line width=0.7,xshift=10cm,yshift=5cm,rotate=90-2*63.43494882292198155938](0,0)--(1,0)--(1,1)--(2,1)--(2,2)--(1,2)--(1,3)--(0,3)--(0,4)--(-1,4)--(-1,3)--(0,3)--(0,4)--(1,4)--(1,5)--(0,5)--(0,6)--(-1,6)--(-1,7)--(-2,7)--(-2,6)--(-3,6)--(-3,5)--(-4,5)--(-4,4)--(-3,4);
\draw[-,line width=0.7,xshift=10cm,yshift=10cm,rotate=180-2*63.43494882292198155938](0,0)--(1,0)--(1,1)--(2,1)--(2,2)--(1,2)--(1,3)--(0,3)--(0,4)--(-1,4)--(-1,3)--(0,3)--(0,4)--(1,4)--(1,5)--(0,5)--(0,6)--(-1,6)--(-1,7)--(-2,7)--(-2,6)--(-3,6)--(-3,5)--(-4,5)--(-4,4)--(-3,4);
\end{tikzpicture}
\\
\vdots
&
\vdots
\\
&
\\
\quad\text{\begin{tiny}Construction of the $(+1,+1,+1,-1,+1)$-Koch curve.\end{tiny}}\quad
&
\quad\text{\begin{tiny}Construction of the $(+1,+1,-1,-1,-1)$-Koch curve.\end{tiny}}\quad\
\end{tabular}
$$

$$$$

It is well known that the classical Koch curve has Hausdorff, packing and box dimension $\log4/\log3$ since the corresponding IFS satisfies the open set condition (OSC). As a generalization, we have the following, where $\lfloor x\rfloor$ denotes the greatest integer no larger than $x$.

\begin{corollary}\label{cor} Let $m\ge2$ be an integer, $\delta_0=\cdots=\delta_{\lfloor\frac{m}{4}\rfloor}=+1$, $\delta_{\lfloor\frac{m}{4}\rfloor+1}=\cdots=\delta_{m-\lfloor\frac{m}{4}\rfloor-1}=-1$, $\delta_{m-\lfloor\frac{m}{4}\rfloor}=\cdots=\delta_m=+1$ and $\delta=(\delta_n)_{n\ge0}$ be the $(+1,\delta_1,\cdots,\delta_m)$-Thue-Morse sequence. Then $p(m+1)$ is a real number in $[3,m+1]$, the $(+1,\delta_1,\cdots,\delta_m)$-IFS satisfies the OSC, and the $(+1,\delta_1,\cdots,\delta_m)$-Koch curve has Hausdorff, packing and box dimension $\log(m+1)/\log p(m+1)$.
\end{corollary}

To obtain the Hausdorff dimension of the $(+1,\delta_1,\cdots,\delta_m)$-Koch curve in Corollary \ref{cor}, one can try to use the last statement in Theorem \ref{main}. But here we use classical theory on IFS by verifying the OSC.

\begin{remark} Generalized Thue-Morse sequences defined in this paper are essentially contained in the concept of generalized Morse sequences in \cite{K68}. In fact, given $\delta_1,\cdots,\delta_m\in\{-1,+1\}$, for the $(+1,\delta_1,\cdots,\delta_m)$-Thue-Morse sequence $\delta=(\delta_n)_{n\ge0}$, if we write $\delta_n=(-1)^{\theta_n}$ where $\theta=(\theta_n)_{n\ge0}$ is a sequence on $\{0,1\}$, by \cite[Proposition 3.1 (1)]{L20} and inductive, one can check that
$$\theta=(0,\theta_1,\cdots,\theta_m)\times(0,\theta_1,\cdots,\theta_m)\times(0,\theta_1,\cdots,\theta_m)\times\cdots$$
where we use the notation of products of blocks mentioned in \cite{K68}. It follows from \cite[Lemma 1]{K68} that $\theta$ is periodic if and only if $\theta=0^\infty$ or $(01)^\infty$. Therefore, if $\theta$ is not the trivial $0^\infty$ or $(01)^\infty$, it is a generalized Morse sequence in the sense of \cite{K68}, and $\delta$ can be viewed as a $\pm1$ version of such a sequence.
\end{remark}

We give some notation and preliminaries in Section 2, and then prove Theorem \ref{main} and Corollary \ref{cor} in Section 3.

\section{Notation and preliminaries}

For any $z_1,z_2\in\C$, we use $[z_1,z_2]:=\{cz_1+(1-c)z_2:c\in[0,1]\}$ to denote the segment connecting $z_1$ and $z_2$. For any $c\in\C$ and $Z\subset\C$, let $cZ:=\{cz:z\in Z\}$ and $c+Z:=\{c+z:z\in Z\}$. Besides, for any $z\in\C$ we use $\text{Re }z$ and $\text{Im }z$ to denote respectively the real part and the imaginary part of $z$.

Let $\cA$ be a finite alphabet of symbols and $\cA^*$ be the free monoid generated by $\cA$.

A map $\phi:\cA^*\to\cA^*$ is called a \textit{morphism} if
$$\phi(uv)=\phi(u)\phi(v)$$
for all words $u,v\in\cA^*$. Moreover $\phi$ is called \textit{null-free} if $\phi(a)$ is not the empty word for any $a\in\cA$, and called \textit{primitive} if there exists an $n\in\N$ such that $a\in\phi^n(b)$ for all $a,b\in\cA$, where $u\in v$ denotes that $u$ occurs in $v$ for any words $u,v\in\cA^*$. For a morphism $\phi:\cA^*\to\cA^*$, the corresponding matrix $M_\phi=(m_{a,b})_{a,b\in\cA}$ is defined by $m_{a,b}:=|\phi(a)|_b$, where $|w|_b$ denotes the number of the symbol $b$ in the word $w$. In addition, we use $|w|$ to denote the length of the finite word $w$.

A map $f:\cA^*\to\C$ is called a \textit{homomorphism} if
$$f(uv)=f(u)+f(v)$$
for all words $u,v\in \cA^*$, and an $\R$-linear map $L:\C\to\C$ (regarded as $\R^2\to\R^2$) is called \textit{expansive} if both eigenvalues have modulus more than one.

Let $\cH(\C)$ be the set of all non-empty compact subsets of $\C$ and $d_H$ be the Hausdorff metric on $\cH(\C)$. The following result was given by Dekking.

\begin{theorem}(\cite[Theorem 2.4]{D82a})\label{converge} Let $\phi:\cA^*\to\cA^*$ be a null-free morphism, $f:\cA^*\to\C$ be a homomorphism, $L:\C\to\C$ be an expansive $\R$-linear map such that
$$f\circ\phi=L\circ f,$$
and $K:\cA^*\to\cH(\C)$ be a map satisfying
$$K(uv)=K(u)\cup(f(u)+K(v))$$
for all $u,v\in \cA^*$. Then for any non-empty word $w\in\cA^*$, there exists a unique compact set $W$ such that
$$L^{-n}K(\phi^n(w))\overset{d_H}{\longrightarrow}W\quad\text{as }n\to\infty,$$
and $W$ is a continuous image of $[0,1]$.
\end{theorem}

In the following we recall some preliminaries on iterated function systems. A map $S:\C\to\C$ is called a \textit{contraction} if there exists $c\in(0,1)$ such that
$$|S(z_1)-S(z_2)|\le c|z_1-z_2|\quad\text{for all }z_1,z_2\in\C.$$
Moreover, if equality holds, i.e., if $|S(z_1)-S(z_2)|=c|z_1-z_2|$ for all $z_1,z_2\in\C$, we say that $S$ is a \textit{contracting similarity}.

A finite family of contractions $\{S_1,S_2,\cdots,S_n\}$, with $n\ge2$, is called an \textit{iterated function system} (\textit{IFS}). The following is a fundamental result. See for example \cite[Theorem 9.1]{F90}.

\begin{theorem}\label{attractor} Any family of contractions $\{S_1,\cdots,S_n\}$ has a unique attractor $F$, i.e., a non-empty compact set such that
$$F=\bigcup_{j=1}^nS_j(F).$$
\end{theorem}

We say that an IFS $\{S_1,\cdots,S_n\}$ satisfies the \textit{open set condition} (\textit{OSC}) if there exists a non-empty bounded open set $V$ such that
$$\bigcup_{j=1}^nS_j(V)\subset V$$
with the union disjoint. The following theorem is well known. See for example \cite[Theorem 9.3]{F90}.

\begin{theorem}\label{OSC} If the OSC holds for the contracting similarities $S_j:\C\to\C$ with the ratios $c_j\in(0,1)$ for all $j\in\{1,\cdots,n\}$, then the attractor of the IFS $\{S_1,\cdots,S_n\}$ has Hausdorff, packing and box dimension $s$, where $s$ is given by
$$\sum_{j=1}^nc_j^s=1.$$
\end{theorem}

To end this section, we present the following basic property for contractions.

\begin{proposition}\label{dH} Let $S_1,S_2,\cdots,S_n$ be contractions on $\C$. Write
$$S(A):=\bigcup_{j=1}^nS_j(A)\quad\text{for all }A\subset\C.$$
Then for all $F,F_1,F_2,\cdots\subset\C$ such that $F_k\overset{d_H}{\longrightarrow}F$ as $k\to\infty$, we have $S(F_k)\overset{d_H}{\longrightarrow}S(F)$.
\end{proposition}
\begin{proof} This follows from the fact that for all $k\in\N$ we have
$$d_H(S(F_k),S(F))\le\max_{1\le j\le n}d_H(S_j(F_k),S_j(F))\le\max_{1\le j\le n}c_jd_H(F_k,F),$$
where for each $j\in\{1,\cdots,n\}$, $c_j\in(0,1)$  satisfies $|S_j(z_1)-S_j(z_2)|\le c_j|z_1-z_2|$ for all $z_1,z_2\in\C$.
\end{proof}

\section{Proofs of Theorem \ref{main} and Corollary \ref{cor}}

\begin{proof}[Proof of Theorem \ref{main}] Let $m\in\N$, $\delta_0=+1$, $\delta_1,\cdots,\delta_m\in\{+1,-1\}$ and $\delta=(\delta_n)_{n\ge0}$ be the $(+1,\delta_1,\cdots,\delta_m)$-Thue-Morse sequence such that $|p(m+1)|>1$.
\newline(1) Prove that there exists a unique compact set $K\subset\C$ such that
$$(p(m+1))^{-n}P(n)\overset{d_H}{\longrightarrow}K\quad\text{as }n\to\infty$$
and $K$ is a continuous image of $[0,1]$ by using Theorem \ref{converge}.
\newline\textcircled{\scriptsize{1}} If $m$ is odd, let $\cA:=\{0,1,2,\cdots,2m-1\}$. Define the morphism $\phi:\cA^*\to \cA^*$ by
$$a\mapsto d_{a,0}d_{a,1}\cdots d_{a,m}$$
for all $a\in \cA$ where
$$d_{a,k}:=\left\{\begin{array}{lll}
a+2k&\text{mod }2m& \text{if } \delta_k=+1\\
a+2k+m&\text{mod }2m& \text{if } \delta_k=-1
\end{array}\right.$$
for all $k\in\{0,1,\cdots,m\}$. Obviously $d_{a,0}=a$ for all $a\in \cA$ and it is straightforward to check
$$e^\frac{d_{a,k}\pi i}{m}=\delta_ke^\frac{(a+2k)\pi i}{m}$$
for all $k\in\{0,1,\cdots,m\}$. Let $\epsilon$ be the empty word. Define $f(\epsilon):=0$ and
$$f(w_1\cdots w_n):=\sum_{k=1}^ne^\frac{w_k\pi i}{m}$$
for any $w_1\cdots w_n\in\cA^*$. Then $f:\cA^*\to\C$ is a homomorphism satisfying
$$f(a)=e^{\frac{a\pi i}{m}}$$
for all $a\in \cA$ and
$$f(uv)=f(u)+f(v)$$
for all $u,v\in \cA^*$. Let $L:\C\to\C$ be the linear map defined by
$$L(z):=p(m+1)\cdot z$$
for all $z\in\C$. It follows from $|p(m+1)|>1$ that $L$ is expansive.

We can check $f\circ\phi=L\circ f$. In fact, for the empty word we have $f\circ\phi(\epsilon)=f(\epsilon)=0=L(0)=L\circ f(\epsilon)$, for any $a\in \cA$ we have
\begin{eqnarray*}
f\circ\phi(a)&=&f(d_{a,0}\cdots d_{a,m})=\sum_{k=0}^me^\frac{d_{a,k}\pi i}{m}=\sum_{k=0}^m\delta_ke^\frac{(a+2k)\pi i}{m}=e^\frac{a\pi i}{m}\sum_{k=0}^m\delta_ke^\frac{2k\pi i}{m}\\
&=&f(a)p(m+1)=L\circ f(a),
\end{eqnarray*}
and for any $w_1\cdots w_n\in \cA^*$ we have
\begin{eqnarray*}
f\circ\phi(w_1\cdots w_n)&=&f(\phi(w_1)\cdots\phi(w_n))=f(\phi(w_1))+\cdots+f(\phi(w_n))\\
&=&L(f(w_1))+\cdots+L(f(w_n))=L(f(w_1)+\cdots+f(w_n))=L\circ f(w_1\cdots w_n).
\end{eqnarray*}

Define $K(\epsilon)$ to be the singleton $\{0\}$,
$$K(a):=[0,f(a)]$$
for any $a\in \cA$, and
$$K(w_1\cdots w_n):=\bigcup_{k=1}^n\Big(f(w_1\cdots w_{k-1})+K(w_k)\Big)$$
for any $w_1\cdots w_n\in\cA^*$, where $f(w_1\cdots w_{k-1})$ is regarded as $0$ for $k=1$. Then $K:\cA^*\to\cH(\C)$ satisfies
$$K(uv)=K(u)\cup(f(u)+K(v))$$
for all $u,v\in \cA^*$. Now applying Theorem \ref{converge}, there exists a unique compact set $K\subset\C$ such that
$$(p(m+1))^{-n}K(\phi^n(0))\overset{d_H}{\longrightarrow}K\quad\text{as }n\to\infty,$$
and $K$ is a continuous image of $[0,1]$. In the following we only need to check $K(\phi^n(0))=P(n)$ for all $n\in\N_0$.
\begin{itemize}
\item[i)] First we prove that for all $a\in \cA$, $j\in\{1,2,\cdots,m\}$ and $n\in\{0,1,2,\cdots\}$ we have
\begin{eqnarray}\label{f}
f(\phi^n(d_{a,0}\cdots d_{a,j-1}))=e^{\frac{a\pi i}{m}}p(j(m+1)^n)
\end{eqnarray}
by induction on $n$. In fact, for $n=0$ we have
$$f(d_{a,0}\cdots d_{a,j-1})=\sum_{k=0}^{j-1}e^\frac{d_{a,k}\pi i}{m}=\sum_{k=0}^{j-1}\delta_ke^\frac{(a+2k)\pi i}{m}=e^\frac{a\pi i}{m}p(j).$$
Suppose that (\ref{f}) is true for some $n\ge0$. Then for $n+1$, on the one hand
$$f(\phi^{n+1}(d_{a,0}\cdots d_{a,j-1}))=L(f(\phi^n(d_{a,0}\cdots d_{a,j-1})))=p(m+1)e^\frac{a\pi i}{m}\sum_{r=0}^{j(m+1)^n-1}\delta_re^\frac{2r\pi i}{m}$$
where the first equality follows from $f\circ\phi=L\circ f$ and the second equality follows from the definition of $L$ and the inductive hypothesis, and on the other hand
$$e^\frac{a\pi i}{m}p(j(m+1)^{n+1})=e^\frac{a\pi i}{m}\sum_{k=0}^{j(m+1)^{n+1}-1}\delta_ke^\frac{2k\pi i}{m}=e^\frac{a\pi i}{m}\sum_{r=0}^{j(m+1)^n-1}\sum_{k=r(m+1)}^{r(m+1)+m}\delta_ke^\frac{2k\pi i}{m}.$$
It suffices to check
$$p(m+1)\delta_re^\frac{2r\pi i}{m}=\sum_{k=r(m+1)}^{r(m+1)+m}\delta_ke^\frac{2k\pi i}{m}$$
for all $r\in\{0,1,\cdots,j(m+1)^n-1\}$. In fact we have
$$\sum_{k=r(m+1)}^{r(m+1)+m}\delta_ke^\frac{2k\pi i}{m}=\sum_{k=0}^m\delta_{r(m+1)+k}e^\frac{2(r(m+1)+k)\pi i}{m}=\sum_{k=0}^m\delta_r\delta_ke^\frac{2r\pi i}{m}e^\frac{2k\pi i}{m}=p(m+1)\delta_re^\frac{2r\pi i}{m},$$
where the second equality follows from $\delta_{r(m+1)+k}=\delta_r\delta_k$ (see \cite[Proposition 3.1 (1)]{L20}).
\item[ii)] To check $K(\phi^n(0))=P(n)$ for all $n\in\N_0$, it suffices to prove
\begin{eqnarray}\label{K}
K(\phi^n(a))=e^\frac{a\pi i}{m}P(n)\quad\text{for all }a\in\cA
\end{eqnarray}
by induction on $n$. In fact, for $n=0$ we have
$$K(a)=[0,e^\frac{a\pi i}{m}]=e^\frac{a\pi i}{m}[0,1]=e^\frac{a\pi i}{m}P(0).$$
Suppose that (\ref{K}) is true for some $n\ge0$. Then for $n+1$, on the one hand
\begin{eqnarray*}
K(\phi^{n+1}(a))&=&K(\phi^n(d_{a,0}\cdots d_{a,m}))\\
&=&\bigcup_{j=0}^m\Big(f(\phi^n(d_{a,0}\cdots d_{a,j-1}))+K(\phi^n(d_{a,j}))\Big)\\
& &\big(\text{where }f(\phi^n(d_{a,0}\cdots d_{a,j-1}))\text{ is regarded as }0\text{ for }j=0\big)\\
&\overset{(*)}{=}&\bigcup_{j=0}^m\Big(e^\frac{a\pi i}{m}p(j(m+1)^n)+e^\frac{d_{a,j}\pi i}{m}P(n)\Big)\\
&=&\bigcup_{j=0}^m\Big(e^\frac{a\pi i}{m}p(j(m+1)^n)+\delta_je^\frac{(a+2j)\pi i}{m}P(n)\Big)\\
&=&e^\frac{a\pi i}{m}\bigcup_{j=0}^m\Big(p(j(m+1)^n)+\delta_je^\frac{2j\pi i}{m}P(n)\Big)
\end{eqnarray*}
where ($*$) follows from the inductive hypothesis and (\ref{f}), and on the other hand
\begin{eqnarray*}
P(n+1)&=&\bigcup_{k=1}^{(m+1)^{n+1}}[p(k-1),p(k)]\\
&=&\bigcup_{j=0}^m\bigcup_{k=j(m+1)^n+1}^{(j+1)(m+1)^n}[p(k-1),p(k)]\\
&=&\bigcup_{j=0}^m\bigcup_{k=1}^{(m+1)^n}[p(j(m+1)^n+k-1),p(j(m+1)^n+k)]\\
&=&\bigcup_{j=0}^m\bigcup_{k=1}^{(m+1)^n}\big[\sum_{r=0}^{j(m+1)^n+k-2}\delta_re^\frac{2r\pi i}{m},\sum_{r=0}^{j(m+1)^n+k-1}\delta_re^\frac{2r\pi i}{m}\big]\\
& &\big(\text{where $\sum_{r=a}^b\cdot$ is regarded as }0\text{ if }a>b\big)\\
&=&\bigcup_{j=0}^m\Big(\sum_{r=0}^{j(m+1)^n-1}\delta_re^\frac{2r\pi i}{m}+\bigcup_{k=1}^{(m+1)^n}\big[\sum_{r=j(m+1)^n}^{j(m+1)^n+k-2}\delta_re^\frac{2r\pi i}{m},\sum_{r=j(m+1)^n}^{j(m+1)^n+k-1}\delta_re^\frac{2r\pi i}{m}\big]\Big)\\
&=&\bigcup_{j=0}^m\Big(p(j(m+1)^n)+\bigcup_{k=1}^{(m+1)^n}\big[\sum_{r=0}^{k-2}\delta_{j(m+1)^n+r}e^\frac{2(j(m+1)^n+r)\pi i}{m},\sum_{r=0}^{k-1}\delta_{j(m+1)^n+r}e^\frac{2(j(m+1)^n+r)\pi i}{m}\big]\Big)\\
&\overset{(**)}{=}&\bigcup_{j=0}^m\Big(p(j(m+1)^n)+\bigcup_{k=1}^{(m+1)^n}\big[\sum_{r=0}^{k-2}\delta_j\delta_re^\frac{2(j+r)\pi i}{m},\sum_{r=0}^{k-1}\delta_j\delta_re^\frac{2(j+r)\pi i}{m}\big]\Big)\\
&=&\bigcup_{j=0}^m\Big(p(j(m+1)^n)+\delta_je^\frac{2j\pi i}{m}\bigcup_{k=1}^{(m+1)^n}[p(k-1),p(k)]\Big)\\
&=&\bigcup_{j=0}^m\Big(p(j(m+1)^n)+\delta_je^\frac{2j\pi i}{m}P(n)\Big)
\end{eqnarray*}
where ($**$) follows from $\delta_{j(m+1)^n+r}=\delta_j\delta_r$ (see \cite[Proposition 3.1 (1)]{L20}). Thus $K(\phi^{n+1}(a))=e^\frac{a\pi i}{m}P(n+1)$.
\end{itemize}
\textcircled{\scriptsize{2}} If $m$ is even, let $\cA:=\{0,1,2,\cdots,m-1\}$. Define the morphism $\phi:\cA^*\to \cA^*$ by
$$a\mapsto d_{a,0}d_{a,1}\cdots d_{a,m}$$
for all $a\in \cA$ where
$$d_{a,k}:=\left\{\begin{array}{lll}
a+k&\text{mod }m&\text{if } \delta_k=+1\\
a+k+\frac{m}{2}&\text{mod }m&\text{if } \delta_k=-1
\end{array}\right.$$
for all $k\in\{0,1,\cdots,m\}$. Obviously $d_{a,0}=a$ for all $a\in \cA$ and it is straightforward to check
$$e^\frac{2d_{a,k}\pi i}{m}=\delta_ke^\frac{2(a+k)\pi i}{m}$$
for all $k\in\{0,1,\cdots,m\}$. Define $f(\epsilon):=0$ and
$$f(w_1\cdots w_n):=\sum_{k=1}^ne^\frac{2w_k\pi i}{m}$$
for any $w_1\cdots w_n\in\cA^*$. Then $f:\cA^*\to\C$ is a homomorphism satisfying
$$f(a)=e^{\frac{2a\pi i}{m}}$$
for all $a\in \cA$ and
$$f(uv)=f(u)+f(v)$$
for all $u,v\in \cA^*$. Let $L:\C\to\C$ and $K:\cA^*\to\cH(\C)$ be defined in the same way as \textcircled{\scriptsize{1}}. Then we can prove
\begin{equation}\label{f2}
f(\phi^n(d_{a,0}\cdots d_{a,j-1}))=e^{\frac{2a\pi i}{m}}p(j(m+1)^n)
\end{equation}
for all $j\in\{1,2,\cdots,m\}$, $a\in \cA$ and $n\in\N_0$, and then
$$K(\phi^n(a))=e^\frac{2a\pi i}{m}P(n).$$
Thus $K(\phi^n(0))= P(n)$ for all $n\in\N_0$. By applying Theorem \ref{converge}, there exists a unique compact set $K\subset\C$ such that
$$(p(m+1))^{-n}P(n)\overset{d_H}{\longrightarrow}K\quad\text{as }n\to\infty,$$
and $K$ is a continuous image of $[0,1]$.
\newline(2) Prove that $K$ is the unique attractor of the IFS $\{S_j\}_{0\le j\le m}$.
\newline By Theorem \ref{attractor} it suffices to show $K=\cup_{j=0}^mS_j(K)$. Let $Q_n:=(p(m+1))^{-n}P(n)$ for all $n\in\N_0$. Since $Q_n\overset{d_H}{\longrightarrow}K$ and Proposition \ref{dH} imply $\cup_{j=0}^mS_j(Q_n)\overset{d_H}{\longrightarrow}\cup_{j=0}^mS_j(K)$ as $n\to\infty$, we only need to prove $Q_{n+1}=\cup_{j=0}^mS_j(Q_n)$ for all $n\in\N_0$ in the following. In fact,
\begin{eqnarray*}
Q_{n+1}&=&(p(m+1))^{-(n+1)}P(n+1)\\
&\overset{(*)}{=}&(p(m+1))^{-(n+1)}\bigcup_{j=0}^m\Big(p(j(m+1)^n)+\delta_je^\frac{2j\pi i}{m}P(n)\Big)\\
&=&\bigcup_{j=0}^m\Big((p(m+1))^{-(n+1)}p(j(m+1)^n)+(p(m+1))^{-(n+1)}\delta_je^\frac{2j\pi i}{m}P(n)\Big)\\
&\overset{(**)}{=}&\bigcup_{j=0}^m\Big((p(m+1))^{-1}p(j)+(p(m+1))^{-(n+1)}\delta_je^\frac{2j\pi i}{m}P(n)\Big)\\
&=&\bigcup_{j=0}^m(p(m+1))^{-1}\Big(p(j)+\delta_je^\frac{2j\pi i}{m}Q_n\Big)\\
&=&\bigcup_{j=0}^mS_j(Q_n),
\end{eqnarray*}
where ($*$) follows from the recurrence relation between $P(n+1)$ and $P(n)$ deduced at the end of the proof in (1) \textcircled{\scriptsize{1}} (noting that this relation is true no matter $m$ is odd or even), and ($**$) follows from $p(0)=0$ and
$$p(j(m+1)^n)\xlongequal[\text{and (\ref{f2})}]{\text{by (\ref{f})}}f(\phi^n(d_{0,0}\cdots d_{0,j-1}))\xlongequal[]{f\circ\phi=L\circ f}L^n(f(d_{0,0}\cdots d_{0,j-1}))\xlongequal[\text{and (\ref{f2})}]{\text{by (\ref{f})}}(p(m+1))^np(j)$$
for all $j\in\{1,2,\cdots,m\}$ and $n\in\N_0$.
\newline(3) Prove that
$$\dim_HK=\frac{\log(m+1)}{\log|p(m+1)|}$$
if and only if there exists $\varepsilon>0$ such that
$$\varliminf_{n\to\infty}\frac{\cL((P(n))^\varepsilon)}{(m+1)^n}>0.$$
Noting that $|\phi^n(0)|=(m+1)^n$, $K(\phi^n(0))=P(n)$ is proved in (1) and $L:\C\to\C$ (regarded as $\R^2\to\R^2$) is a similarity with eigenvalues of the same modulus $|p(m+1)|>1$, by applying \cite[Dekking's conjecture]{B86} (which was proved), we only need to check that the eigenvalue of $M_\phi$ (the corresponding matrix of $\phi$) with greatest modulus is $m+1$ and $\phi$ is primitive. Note that according to whether $m$ is odd or even, the definition of $\phi$ in (1) is different.
\newline\textcircled{\scriptsize{1}} If $m$ is odd, recall $\cA:=\{0,1,2,\cdots,2m-1\}$. On the calculation between the symbols in $\cA$, we consider the mod $2m$ congruence class (for example $5+(2m-3)=2$). Recall the definition of $\phi$. For any $a,b\in\cA$, the equivalences of $d_{a,k}=b$ and $d_{a+1,k}=b+1$ for all $k\in\{0,1,\cdots,m\}$ imply $|\phi(a)|_b=|\phi(a+1)|_{b+1}$. This means that $M_\phi$ is a circulant matrix, and the eigenvalue with greatest modulus is $|\phi(0)|_0+|\phi(0)|_1+\cdots+|\phi(0)|_{2m-1}=|\phi(0)|=m+1$.

In the following we prove that $\phi$ is primitive. That is, there exists $n\in\N$ such that $b\in\phi^n(a)$ for all $a,b\in\cA$, where $u\in v$ means that $u$ occurs in $v$ for any words $u,v\in\cA^*$. For any word $w=w_1\cdots w_k\in \cA^*$ and any symbol $a\in \cA$, write
$$w+a=w_1\cdots w_k+a:=(w_1+a)\cdots(w_k+a).$$
Then we have
\begin{equation}\label{+}
\phi(w+a)=\phi(w_1+a)\cdots\phi(w_k+a)=(\phi(w_1)+a)\cdots(\phi(w_k)+a)=\phi(w)+a,
\end{equation}
where the second equality follows from
$$\phi(b+a)=d_{b+a,0}d_{b+a,1}\cdots d_{b+a,m}=(d_{b,0}+a)(d_{b,1}+a)\cdots(d_{b,m}+a)=\phi(b)+a$$
for any $a,b\in\cA$. By applying (\ref{+}) consecutively, for all $a\in\cA$ and $n\in\N$ we have
\begin{equation}\label{n+}
\phi^n(a)=\phi^{n-1}(\phi(0+a))=\phi^{n-1}(\phi(0)+a)=\phi^{n-2}(\phi(\phi(0)+a))=\phi^{n-2}(\phi^2(0)+a)=\cdots=\phi^n(0)+a,
\end{equation}
and then $b\in\phi^n(a)$ is equivalent to $b-a\in\phi^n(0)$ for all $b\in\cA$. Thus we only need to prove that there exists $n\in\N$ such that $a\in\phi^n(0)$ for all $a\in \cA$.
\begin{itemize}
\item[i)] Suppose $\delta_1=+1$. Then
$$d_{0,1}=2,d_{2,1}=4,d_{4,1}=6,\cdots,d_{2m-4,1}=2m-2,$$
which imply
$$2\in\phi(0),4\in\phi(2),6\in\phi(4),\cdots,2m-2\in\phi(2m-4).$$
By iterating $\phi$ we get
\begin{equation}\label{in}
2\in\phi(0),4\in\phi^2(0),6\in\phi^3(0),\cdots,2m-2\in\phi^{m-1}(0)
\end{equation}
one by one. It follows from
\begin{equation}\label{0in}
0\in\phi(0)\in\phi^2(0)\in\cdots\in\phi^{m-1}(0)\in\phi^m(0)
\end{equation}
that $0,2,4,\cdots,2m-2\in\phi^m(0)$. It suffices to prove $1,3,5,\cdots,2m-1\in\phi^m(0)$ in the following. Since $\delta_1=\cdots=\delta_m=+1$ will imply $p(m+1)=1$ (which contradicts $|p(m+1)|>1$), noting $\delta_1=+1$, there exists $l\in\{2,3,\cdots,m\}$ such that $\delta_l=-1$. This implies
$$d_{0,l}=2l+m,d_{2,l}=2l+m+2,d_{4,l}=2l+m+4,\cdots,d_{2m-2,l}=2l+3m-2$$
and then
$$2l+m\in\phi(0),2l+m+2\in\phi(2),2l+m+4\in\phi(4),\cdots,2l+3m-2\in\phi(2m-2).$$
It follows from (\ref{in}) that
$$2l+m\in\phi(0),2l+m+2\in\phi^2(0),2l+m+4\in\phi^3(0),\cdots,2l+3m-2\in\phi^m(0).$$
By (\ref{0in}) we get $2l+m,2l+m+2,2l+m+4,\cdots,2l+3m-2\in\phi^m(0)$, which is equivalent to $1,3,5,\cdots,2m-1\in\phi^m(0)$. Therefore $a\in\phi^m(0)$ for all $a\in \cA$.
\item[ii)] Suppose $\delta_1=-1$. Then $d_{0,1}=m+2$. By $m+2\in\phi(0)$, we get
$$2(m+2)=m+2+m+2\in\phi(0)+m+2\xlongequal[]{\text{by (\ref{+})}}\phi(m+2)\in\phi^2(0).$$
In the same way we get $3(m+2)\in\phi^3(0),4(m+2)\in\phi^4(0),\cdots,(2m-1)(m+2)\in\phi^{2m-1}(0)$. It follows from
$$0\in\phi(0)\in\phi^2(0)\in\cdots\in\phi^{2m-1}(0)$$
that
\begin{equation}\label{class}
0,m+2,2(m+2),3(m+2),\cdots,(2m-1)(m+2)\in\phi^{2m-1}(0).
\end{equation}
Since $m$ is odd, we know that $m+2$ and $2m$ are relatively prime. This implies that $0,m+2,2(m+2),3(m+2),\cdots,(2m-1)(m+2)$ construct a complete residue system mod $2m$. By (\ref{class}) we get $0,1,2,\cdots,2m-1\in\phi^{2m-1}(0)$.
\end{itemize}
\textcircled{\scriptsize{2}} If $m$ is even, recall $\cA:=\{0,1,2,\cdots,m-1\}$. On the calculation between the symbols in $\cA$, we consider the mod $m$ congruence class (for example $5+(m-3)=2$). Recall the definition of $\phi$. In the same way as \textcircled{\scriptsize{1}}, we know that the eigenvalue of $M_\phi$ with greatest modulus is $m+1$.

In the following it suffices to prove that $\phi$ is primitive. In the same way as \textcircled{\scriptsize{1}}, we get
\begin{equation}\label{even+}
\phi^n(a)=\phi^n(0)+a\quad\text{for all }a\in\cA\text{ and }n\in\N,
\end{equation}
and we only need to prove that there exists $n\in\N$ such that $a\in\phi^n(0)$ for all $a\in \cA$.
\begin{itemize}
\item[i)] Suppose $\delta_1=+1$. Then
$$d_{0,1}=1,d_{1,1}=2,d_{2,1}=3,\cdots,d_{m-2,1}=m-1$$
which imply
$$1\in\phi(0),2\in\phi(1),3\in\phi(2),\cdots,m-1\in\phi(m-2).$$
By iterating $\phi$ we get
$$1\in\phi(0),2\in\phi^2(0),3\in\phi^3(0),\cdots,m-1\in\phi^{m-1}(0)$$
one by one. It follows from
$$0\in\phi(0)\in\phi^2(0)\in\cdots\in\phi^{m-1}(0)$$
that $0,1,2,\cdots,m-1\in\phi^{m-1}(0)$.
\item[ii)] Suppose $\delta_1=-1$. Then $d_{0,1}=\frac{m}{2}+1$ and $d_{\frac{m}{2}+1,1}=2$, which imply $\frac{m}{2}+1\in\phi(0)$ and $2\in\phi(\frac{m}{2}+1)$. It follows from $\phi(\frac{m}{2}+1)\in\phi^2(0)$ that $2\in\phi^2(0)$, and then $\phi^2(2)\in\phi^4(0)$. Since (\ref{even+}) implies $\phi^2(2)=\phi^2(0)+2$, we get $4\in\phi^4(0)$. Repeating this process we get
\begin{equation}\label{246}
2\in\phi^2(0),4\in\phi^4(0),6\in\phi^6(0),\cdots,m-2\in\phi^{m-2}(0).
\end{equation}
It follows from $0\in\phi^2(0)\in\phi^4(0)\in\cdots\in\phi^{m-2}(0)$ that
\begin{equation}\label{024}
0,2,4,\cdots,m-2\in\phi^{m-2}(0).
\end{equation}

First we prove that there exits an odd $a\in\cA$ such that $a\in\phi(0)$ by contradiction. Assume $a\notin\phi(0)$ for all odd $a\in \cA$. By $\phi(0)=d_{0,0}d_{0,1}\cdots d_{0,m}$ we know that $d_{0,0},d_{0,1},\cdots,d_{0,m}$ are all even. Then $d_{0,1}=\frac{m}{2}+1$ implies that $\frac{m}{2}$ is odd. By
$$d_{0,k}:=\left\{\begin{array}{ll}
k & \text{if } \delta_k=+1\\
k+\frac{m}{2} & \text{if } \delta_k=-1
\end{array}\right.$$
for all $k\in\{0,1,\cdots,m\}$, we get
$$\delta_0=\delta_2=\delta_4=\cdots=\delta_m=+1\quad\text{and}\quad\delta_1=\delta_3=\cdots=\delta_{m-1}=-1.$$
It follows that
$$p(m+1)=\sum_{k=0}^m(-1)^ke^\frac{2k\pi i}{m}=1+\sum_{k=1}^\frac{m}{2}(-1)^ke^\frac{2k\pi i}{m}+\sum_{k=\frac{m}{2}+1}^m(-1)^ke^\frac{2k\pi i}{m},$$
where
$$\sum_{k=\frac{m}{2}+1}^m(-1)^ke^\frac{2k\pi i}{m}=\sum_{k=1}^\frac{m}{2}(-1)^{\frac{m}{2}+k}e^\frac{2(\frac{m}{2}+k)\pi i}{m}=\sum_{k=1}^\frac{m}{2}(-1)^{k+1}e^{\pi i}e^\frac{2k\pi i}{m}=\sum_{k=1}^\frac{m}{2}(-1)^ke^\frac{2k\pi i}{m}$$
and
\begin{eqnarray*}
\sum_{k=1}^\frac{m}{2}(-1)^ke^\frac{2k\pi i}{m}&=&\sum_{j=0}^{\frac{1}{2}(\frac{m}{2}-1)}(-1)^{2j+1}e^\frac{2(2j+1)\pi i}{m}+\sum_{j=1}^{\frac{1}{2}(\frac{m}{2}-1)}(-1)^{2j}e^\frac{2(2j)\pi i}{m}\\
&=&\sum_{j=0}^{\frac{1}{2}(\frac{m}{2}-1)}e^{\frac{2(2j+1)\pi i}{m}+\pi i}+\sum_{j=1}^{\frac{1}{2}(\frac{m}{2}-1)}e^\frac{4j\pi i}{m}\\
&=&\sum_{j=0}^{\frac{1}{2}(\frac{m}{2}-1)}e^\frac{4(\frac{1}{2}(\frac{m}{2}+1)+j)\pi i}{m}+\sum_{j=1}^{\frac{1}{2}(\frac{m}{2}-1)}e^\frac{4j\pi i}{m}\\
&=&\sum_{j=\frac{1}{2}(\frac{m}{2}+1)}^{\frac{m}{2}}e^\frac{4j\pi i}{m}+\sum_{j=1}^{\frac{1}{2}(\frac{m}{2}-1)}e^\frac{4j\pi i}{m}=\sum_{j=1}^\frac{m}{2}e^\frac{2j\pi i}{\frac{m}{2}}=0.
\end{eqnarray*}
This implies $p(m+1)=1$, which contradicts $|p(m+1)|>1$. Thus there must exist an odd $a\in \cA$ such that $a\in\phi(0)$, which implies
$$\phi^2(a)\in\phi^3(0),\phi^4(a)\in\phi^5(0),\cdots,\phi^{m-2}(a)\in\phi^{m-1}(0).$$
It follows from $\phi(0)\in\phi^3(0)\in\phi^5(0)\in\cdots\in\phi^{m-1}(0)$ that
\begin{equation}\label{a24}
a,\phi^2(a),\phi^4(a),\cdots,\phi^{m-2}(a)\in\phi^{m-1}(0).
\end{equation}
Since (\ref{even+}) implies
$$\phi^2(a)=\phi^2(0)+a,\phi^4(a)=\phi^4(0)+a,\cdots,\phi^{m-2}(a)=\phi^{m-2}(0)+a,$$
by (\ref{246}) we get
$$a+2\in\phi^2(a),a+4\in\phi^4(a),\cdots,a+m-2\in\phi^{m-2}(a).$$
It follows from (\ref{a24}) that $a,a+2,a+4,\cdots,a+m-2\in\phi^{m-1}(0)$. Recalling that $a$ is odd, we get $1,3,5,\cdots,m-1\in\phi^{m-1}(0)$. Since $0\in\phi(0)$ implies $\phi^{m-2}(0)\in\phi^{m-1}(0)$, by (\ref{024}) we get $0,2,4,\cdots,m-2\in\phi^{m-1}(0)$. Therefore $0,1,2,3,\cdots,m-1\in\phi^{m-1}(0)$.
\end{itemize}
\end{proof}

\begin{proof}[Proof of Corollary \ref{cor}] Let $m\ge2$ be an integer, $\delta_0=\cdots=\delta_{\lfloor\frac{m}{4}\rfloor}=+1$, $\delta_{\lfloor\frac{m}{4}\rfloor+1}=\cdots=\delta_{m-\lfloor\frac{m}{4}\rfloor-1}=-1$, $\delta_{m-\lfloor\frac{m}{4}\rfloor}=\cdots=\delta_m=+1$ and $\delta=(\delta_n)_{n\ge0}$ be the $(+1,\delta_1,\cdots,\delta_m)$-Thue-Morse sequence.
\newline(1) Prove $3\le p(m+1)\le m+1$. In fact, by
$$p(m+1)=\sum_{k=0}^m\delta_ke^\frac{2k\pi i}{m}=\sum_{k=0}^m\delta_k\cos\frac{2k\pi}{m}+i\sum_{k=0}^m\delta_k\sin\frac{2k\pi}{m},$$
it suffices to consider the following \textcircled{\scriptsize{1}} and \textcircled{\scriptsize{2}}.
\begin{itemize}
\item[\textcircled{\scriptsize{1}}] We have $\sum_{k=0}^m\delta_k\sin\frac{2k\pi}{m}=0$ since for all $k\in\{0,1,\cdots,\lfloor\frac{m}{2}\rfloor\}$,
$$\delta_k\sin\frac{2k\pi}{m}+\delta_{m-k}\sin\frac{2(m-k)\pi}{m}\xlongequal[]{\delta_k=\delta_{m-k}}\delta_k(\sin\frac{2k\pi}{m}+\sin(2\pi-\frac{2k\pi}{m}))=0.$$
\item[\textcircled{\scriptsize{2}}] Prove $3\le\sum_{k=0}^m\delta_k\cos\frac{2k\pi}{m}\le m+1$.
\newline Since $\delta_k\cos\frac{2k\pi}{m}=1$ for $k\in\{0,m\}$ and $\delta_k\cos\frac{2k\pi}{m}\le1$ for $k\in\{1,2\cdots,m-1\}$, we only need to check $\sum_{k=1}^{m-1}\delta_k\cos\frac{2k\pi}{m}\ge1$. It suffices to consider the following i) and ii).
\newline i) Prove $\delta_k\cos\frac{2k\pi}{m}\ge0$ for all $k\in\{1,\cdots,m-1\}$.
\begin{itemize}
\item[\textcircled{\scriptsize{a}}] If $0\le k\le\lfloor\frac{m}{4}\rfloor$, we have $\delta_k=+1$ and $0\le\frac{2k\pi}{m}\le\frac{\pi}{2}$.
\item[\textcircled{\scriptsize{b}}] If $\lfloor\frac{m}{4}\rfloor+1\le k\le m-\lfloor\frac{m}{4}\rfloor-1$, we have $\delta_k=-1$ and $\frac{\pi}{2}\le\frac{2k\pi}{m}\le\frac{3\pi}{2}$.
\item[\textcircled{\scriptsize{c}}] If $m-\lfloor\frac{m}{4}\rfloor\le k\le m$, we have $\delta_k=+1$ and $\frac{3\pi}{2}\le\frac{2k\pi}{m}\le2\pi$.
\end{itemize}
\begin{itemize}
\item[ii) \textcircled{\scriptsize{a}}] If $m$ is even, we have $\delta_\frac{m}{2}\cos\frac{2\cdot\frac{m}{2}\cdot\pi}{m}=1$.
\item[\textcircled{\scriptsize{b}}] If $m$ is odd, we have
$$\delta_\frac{m-1}{2}\cos\frac{2\cdot\frac{m-1}{2}\cdot\pi}{m}+\delta_\frac{m+1}{2}\cos\frac{2\cdot\frac{m+1}{2}\cdot\pi}{m}=-\cos(\pi-\frac{\pi}{m})-\cos(\pi+\frac{\pi}{m})=2\cos\frac{\pi}{m}\ge2\cos\frac{\pi}{3}=1.$$
\end{itemize}
\end{itemize}
(2) Since Theorem \ref{main} says that the $(+1,\delta_1,\cdots,\delta_m)$-Koch curve is the unique attractor of the $(+1,\delta_1,\cdots,\delta_m)$-IFS $\{S_j\}_{0\le j\le m}$, to complete the proof, by applying Theorem \ref{OSC}, it suffices to check that $\{S_j\}_{0\le j\le m}$ satisfies the OSC.

When $m=2$, we have $\delta_0=+1,\delta_1=-1,\delta_2=+1$, $p(m+1)=3$, $S_0(z)=\frac{z}{3}$, $S_1(z)=\frac{z}{3}+\frac{1}{3}$ and $S_2(z)=\frac{z}{3}+\frac{2}{3}$ for $z\in\C$, and we can take the open set $\{x+yi:x,y\in(0,1)\}$.

When $m=3$, we have $\delta_0=+1,\delta_1=\delta_2=-1,\delta_3=+1$, $p(m+1)=3$, $S_0(z)=\frac{z}{3}$, $S_1(z)=\frac{1}{3}-\frac{z}{3}e^\frac{2\pi i}{3}$, $S_2(z)=\frac{1}{3}-\frac{1}{3}e^\frac{2\pi i}{3}-\frac{z}{3}e^\frac{4\pi i}{3}$ and $S_3(z)=\frac{z}{3}+\frac{2}{3}$ for $z\in\C$. The attractor of this IFS is exactly the classical Koch curve and this IFS satisfies the OSC, where the open set can be taken by the open isosceles triangle $\{x+yi:x,y\in\R,y<0,x+\sqrt{3}y>0,x-\sqrt{3}y<1\}$.

In the following we consider $m\ge4$. Let
$$a_m:=\sum_{k=0}^{\lfloor\frac{m}{4}\rfloor}\cos\frac{2k\pi}{m}\quad\text{and}\quad b_m:=\sum_{k=0}^{\lfloor\frac{m}{4}\rfloor}\sin\frac{2k\pi}{m}.$$
Then $a_m,b_m>0$ and $p(\lfloor\frac{m}{4}\rfloor+1)=a_m+b_mi$.
\newline\textcircled{\scriptsize{1}} If $m\equiv0,1$ or $2$ mod $4$, define
$$V:=\Big\{x+yi:x,y\in\R,y>0,b_mx-a_my>0,b_mx+a_my<b_m\Big\}.$$
See Figures \ref{0}, \ref{1} and \ref{2}. Obviously $V$ is the non-empty bounded open isosceles triangle with base $[0,1]$ and vertex $\frac{1}{2}+\frac{b_m}{2a_m}i$. Note that for each $j\in\{0,1,\cdots,m\}$, $S_j$ is the composition of the rotation $\delta_je^\frac{2j\pi i}{m}\cdot$, the scaling $(p(m+1))^{-1}\cdot$ and the translation $\cdot+\frac{p(j)}{p(m+1)}$, and $S_j$ maps $[0,1]$ to $[\frac{p(j)}{p(m+1)},\frac{p(j+1)}{p(m+1)}]$. It is straightforward to see that $\{S_j(V)\}_{0\le j\le m}$ are the disjoint open isosceles triangles with bases $\{[\frac{p(j)}{p(m+1)},\frac{p(j+1)}{p(m+1)}]\}_{0\le j\le m}$ and vertexes $\{S_j(\frac{1}{2}+\frac{b_m}{2a_m}i)\}_{0\le j\le m}$ all on the upper side of the polygonal line $\frac{P(1)}{p(m+1)}$. To verify $\cup_{j=0}^mS_j(V)\subset V$, in the following we check Im $p(\frac{m+1}{2})\ge0$ if $m$ is odd and Im $p(\frac{m}{2})\ge0$ if $m$ is even.
\begin{itemize}
\item[i)] If $m$ is odd, by $m\equiv1$ mod $4$, we have $\frac{m-1}{2}=2\lfloor\frac{m}{4}\rfloor$ and then
\begin{eqnarray*}
\text{Im }p(\frac{m+1}{2})&=&\sum_{k=0}^{\lfloor\frac{m}{4}\rfloor}\sin\frac{2k\pi}{m}-\sum_{k=\lfloor\frac{m}{4}\rfloor+1}^{\frac{m+1}{2}-1}\sin\frac{2k\pi}{m}=\sum_{k=1}^{\lfloor\frac{m}{4}\rfloor}\sin\frac{2k\pi}{m}-\sum_{k=\lfloor\frac{m}{4}\rfloor+1}^{2\lfloor\frac{m}{4}\rfloor}\sin\frac{2k\pi}{m}\\
&=&\sum_{k=1}^{\lfloor\frac{m}{4}\rfloor}\sin\frac{2k\pi}{m}-\sum_{k=1}^{\lfloor\frac{m}{4}\rfloor}\sin\frac{2(2\lfloor\frac{m}{4}\rfloor+1-k)\pi}{m}=\sum_{k=1}^{\lfloor\frac{m}{4}\rfloor}(\sin\frac{2k\pi}{m}-\sin\frac{(2k-1)\pi}{m})\ge0.
\end{eqnarray*}
\item[ii)] If $m$ is even and $m\equiv2$ mod $4$, by $\frac{m}{2}-1=2\lfloor\frac{m}{4}\rfloor$, in a way similar to i) we can get Im $p(\frac{m}{2})=0$.
\item[iii)] If $m$ is even and $m\equiv0$ mod $4$, we have
\begin{eqnarray*}
\text{Im }p(\frac{m}{2})&=&\sum_{k=0}^\frac{m}{4}\sin\frac{2k\pi}{m}-\sum_{k=\frac{m}{4}+1}^{\frac{m}{2}-1}\sin\frac{2k\pi}{m}=\sum_{k=1}^{\frac{m}{4}-1}\sin\frac{2k\pi}{m}+\sin\frac{2\cdot\frac{m}{4}\cdot\pi}{m}-\sum_{k=1}^{\frac{m}{4}-1}\sin\frac{2(\frac{m}{2}-k)\pi}{m}\\
&=&1+\sum_{k=1}^{\frac{m}{4}-1}(\sin\frac{2k\pi}{m}-\sin\frac{(m-2k)\pi}{m})=1\ge0.
\end{eqnarray*}
\end{itemize}
\textcircled{\scriptsize{2}} If $m\equiv3$ mod $4$, we have $\frac{m+1}{2}-1=2\lfloor\frac{m}{4}\rfloor+1$ and then
\begin{eqnarray*}
\text{Im }p(\frac{m+1}{2})&=&\sum_{k=0}^{\lfloor\frac{m}{4}\rfloor}\sin\frac{2k\pi}{m}-\sum_{k=\lfloor\frac{m}{4}\rfloor+1}^{2\lfloor\frac{m}{4}\rfloor+1}\sin\frac{2k\pi}{m}=\sum_{k=0}^{\lfloor\frac{m}{4}\rfloor}\sin\frac{2k\pi}{m}-\sum_{k=0}^{\lfloor\frac{m}{4}\rfloor}\sin\frac{2(2\lfloor\frac{m}{4}\rfloor+1-k)\pi}{m}\\
&=&\sum_{k=0}^{\lfloor\frac{m}{4}\rfloor}(\sin\frac{2k\pi}{m}-\sin\frac{(2k+1)\pi}{m})<0.
\end{eqnarray*}
Let
$$c_m:=-\frac{\text{Im }p(\frac{m+1}{2})}{p(m+1)}>0$$
and define
$$V:=\Big\{x+yi:x,y\in\R,b_mx-a_my>0,b_mx+a_my<b_m,2c_mx+y>0,2c_mx-y<2c_m\Big\}.$$
See Figure \ref{3}. Obviously $V$ is the non-empty bounded open quadrilateral containing two isosceles triangles with the same base $[0,1]$ and one has vertex $\frac{1}{2}+\frac{b_m}{2a_m}i$ and the other has vertex $\frac{1}{2}-c_mi$. It is straightforward to see that $\{S_j(V)\}_{0\le j\le m}$ are open quadrilaterals, each containing two isosceles triangles with the same base $[\frac{p(j)}{p(m+1)},\frac{p(j+1)}{p(m+1)}]$ where one triangle has vertex $S_j(\frac{1}{2}+\frac{b_m}{2a_m}i)$ on the upper side of the polygon $\frac{P(1)}{p(m+1)}$ and the other has vertex $S_j(\frac{1}{2}-c_mi)$ on the lower side. By simple geometrical relation we know that $\{S_j(V)\}_{0\le j\le m}$ are all disjoint and contained in $V$.
\end{proof}

\begin{figure}[H]
\begin{tikzpicture}[scale = 15/11.056]
\draw(0.3,0) arc (0:45:0.3);
\draw(1.4,0) arc (0:22.5:0.4);
\draw(1+0.924+0.707+0.383+0,0+0.383+0.707+0.924+0.6) arc (-90:-67.5:0.4);
\draw node[below]at(0,0){$0$};
\draw node[below]at(11.056,0){$1$};
\draw node[font=\tiny]at(0.4,0.15){$\frac{\pi}{4}$};
\draw node[font=\tiny]at(1.5,-0.2){$\frac{2\pi}{m}$};
\draw node[font=\tiny]at(1+0.924+0.707+0.383+0+0.125,0+0.383+0.707+0.924+0.7-0.4){$\frac{2\pi}{m}$};
\draw[-](0,0)--(1,0);
\draw[-](1,0)--(1+0.924,0+0.383);
\draw node[outer sep=0pt,font=\tiny]at(1+0.924+0.17675,0+0.383+0.17675){.};
\draw node[outer sep=0pt,font=\tiny]at(1+0.924+0.17675+0.17675,0+0.383+0.17675+0.17675){.};
\draw node[outer sep=0pt,font=\tiny]at(1+0.924+0.17675+0.17675+0.17675,0+0.383+0.17675+0.17675+0.17675){.};
\draw[-](1+0.924+0.707,0+0.383+0.707)--(1+0.924+0.707+0.383,0+0.383+0.707+0.924);
\draw[-](1+0.924+0.707+0.383,0+0.383+0.707+0.924)--(1+0.924+0.707+0.383+0,0+0.383+0.707+0.924+1);
\draw[-](1+0.924+0.707+0.383+0,0+0.383+0.707+0.924+1)--(1+0.924+0.707+0.383+0+0.383,0+0.383+0.707+0.924+1-0.924);
\draw[-](1+0.924+0.707+0.383+0+0.383,0+0.383+0.707+0.924+1-0.924)--(1+0.924+0.707+0.383+0+0.383+0.707,0+0.383+0.707+0.924+1-0.924-0.707);
\draw node[outer sep=0pt,font=\tiny]at(1+0.924+0.707+0.383+0+0.383+0.707+0.2355,0+0.383+0.707+0.924+1-0.924-0.707-0.09575){.};
\draw node[outer sep=0pt,font=\tiny]at(1+0.924+0.707+0.383+0+0.383+0.707+0.2355+0.2355,0+0.383+0.707+0.924+1-0.924-0.707-0.09575-0.09575){.};
\draw node[outer sep=0pt,font=\tiny]at(1+0.924+0.707+0.383+0+0.383+0.707+0.2355+0.2355+0.2355,0+0.383+0.707+0.924+1-0.924-0.707-0.09575-0.09575-0.09575){.};
\draw[-](1+0.924+0.707+0.383+0+0.383+0.707+0.942,0+0.383+0.707+0.924+1-0.924-0.707-0.383)--(1+0.924+0.707+0.383+0+0.383+0.707+0.942+1,0+0.383+0.707+0.924+1-0.924-0.707-0.383);
\draw node[outer sep=0pt,font=\tiny]at(1+0.924+0.707+0.383+0+0.383+0.707+0.942+1+0.231,0+0.383+0.707+0.924+1-0.924-0.707-0.383+0+0.09575){.};
\draw node[outer sep=0pt,font=\tiny]at(1+0.924+0.707+0.383+0+0.383+0.707+0.942+1+0.231+0.231,0+0.383+0.707+0.924+1-0.924-0.707-0.383+0+0.09575+0.09575){.};
\draw node[outer sep=0pt,font=\tiny]at(1+0.924+0.707+0.383+0+0.383+0.707+0.942+1+0.231+0.231+0.231,0+0.383+0.707+0.924+1-0.924-0.707-0.383+0+0.09575+0.09575+0.09575){.};
\draw[-](1+0.924+0.707+0.383+0+0.383+0.707+0.942+1+0.924,0+0.383+0.707+0.924+1-0.924-0.707)--(1+0.924+0.707+0.383+0+0.383+0.707+0.942+1+0.924+0.707,0+0.383+0.707+0.924+1-0.924);
\draw[-](1+0.924+0.707+0.383+0+0.383+0.707+0.942+1+0.924+0.707,0+0.383+0.707+0.924+1-0.924)--(1+0.924+0.707+0.383+0+0.383+0.707+0.942+1+0.924+0.707+0.383,0+0.383+0.707+0.924+1);
\draw[-](1+0.924+0.707+0.383+0+0.383+0.707+0.942+1+0.924+0.707+0.383,0+0.383+0.707+0.924+1)--(1+0.924+0.707+0.383+0+0.383+0.707+0.942+1+0.924+0.707+0.383+0,0+0.383+0.707+0.924);
\draw[-](1+0.924+0.707+0.383+0+0.383+0.707+0.942+1+0.924+0.707+0.383+0,0+0.383+0.707+0.924)--(1+0.924+0.707+0.383+0+0.383+0.707+0.942+1+0.924+0.707+0.383+0+0.383,0+0.383+0.707);
\draw node[outer sep=0pt,font=\tiny]at(1+0.924+0.707+0.383+0+0.383+0.707+0.942+1+0.924+0.707+0.383+0+0.383+0.17675,0+0.383+0.707-0.17675){.};
\draw node[outer sep=0pt,font=\tiny]at(1+0.924+0.707+0.383+0+0.383+0.707+0.942+1+0.924+0.707+0.383+0+0.383+0.17675+0.17675,0+0.383+0.707-0.17675-0.17675){.};
\draw node[outer sep=0pt,font=\tiny]at(1+0.924+0.707+0.383+0+0.383+0.707+0.942+1+0.924+0.707+0.383+0+0.383+0.17675+0.17675+0.17675,0+0.383+0.707-0.17675-0.17675-0.17675){.};
\draw[-](1+0.924+0.707+0.383+0+0.383+0.707+0.942+1+0.924+0.707+0.383+0+0.383+0.707,0+0.383)--(1+0.924+0.707+0.383+0+0.383+0.707+0.942+1+0.924+0.707+0.383+0+0.383+0.707+0.924,0);
\draw[-](1+0.924+0.707+0.383+0+0.383+0.707+0.942+1+0.924+0.707+0.383+0+0.383+0.707+0.924,0)--(1+0.924+0.707+0.383+0+0.383+0.707+0.942+1+0.924+0.707+0.383+0+0.383+0.707+0.924+1,0);
\draw[dashed](1,0)--(10.056,0);
\draw[dashed](0,0)--(5.528,5.528);
\draw[dashed](5.528,5.528)--(11.056,0);
\draw[dashed](0.5,0.5)--(1,0);
\draw[dashed](1,0)--(1+0.271,0+0.653);
\draw[dashed](1+0.271,0+0.653)--(1+0.924,0+0.383);
\draw[dashed](1+0.924+0.707,0+0.383+0.707)--(1+0.924+0.707-0.271,0+0.383+0.707+0.604);
\draw[dashed](1+0.924+0.707-0.271,0+0.383+0.707+0.604)--(1+0.924+0.707+0.383,0+0.383+0.707+0.924);
\draw[dashed](1+0.924+0.707+0.383,0+0.383+0.707+0.924)--(1+0.924+0.707+0.383-0.5,0+0.383+0.707+0.924+0.5);
\draw[dashed](1+0.924+0.707+0.383+0,0+0.383+0.707+0.924+1)--(1+0.924+0.707+0.383+0+0.653,0+0.383+0.707+0.924+1-0.271);
\draw[dashed](1+0.924+0.707+0.383+0+0.653,0+0.383+0.707+0.924+1-0.271)--(1+0.924+0.707+0.383+0+0.383,0+0.383+0.707+0.924+1-0.924);
\draw[dashed](1+0.924+0.707+0.383+0+0.383,0+0.383+0.707+0.924+1-0.924)--(1+0.924+0.707+0.383+0+0.383+0.707,0+0.383+0.707+0.924+1-0.924+0);
\draw[dashed](1+0.924+0.707+0.383+0+0.383+0.707,0+0.383+0.707+0.924+1-0.924+0)--(1+0.924+0.707+0.383+0+0.383+0.707,0+0.383+0.707+0.924+1-0.924-0.707);
\draw[dashed](11.056-0.5,0.5)--(11.056-1,0);
\draw[dashed](11.056-1,0)--(11.056-1-0.271,0+0.653);
\draw[dashed](11.056-1-0.271,0+0.653)--(11.056-1-0.924,0+0.383);
\draw[dashed](11.056-1-0.924-0.707,0+0.383+0.707)--(11.056-1-0.924-0.707+0.271,0+0.383+0.707+0.604);
\draw[dashed](11.056-1-0.924-0.707+0.271,0+0.383+0.707+0.604)--(11.056-1-0.924-0.707-0.383,0+0.383+0.707+0.924);
\draw[dashed](11.056-1-0.924-0.707-0.383,0+0.383+0.707+0.924)--(11.056-1-0.924-0.707-0.383+0.5,0+0.383+0.707+0.924+0.5);
\draw[dashed](11.056-1-0.924-0.707-0.383-0,0+0.383+0.707+0.924+1)--(11.056-1-0.924-0.707-0.383-0-0.653,0+0.383+0.707+0.924+1-0.271);
\draw[dashed](11.056-1-0.924-0.707-0.383-0-0.653,0+0.383+0.707+0.924+1-0.271)--(11.056-1-0.924-0.707-0.383-0-0.383,0+0.383+0.707+0.924+1-0.924);
\draw[dashed](11.056-1-0.924-0.707-0.383-0-0.383,0+0.383+0.707+0.924+1-0.924)--(11.056-1-0.924-0.707-0.383-0-0.383-0.707,0+0.383+0.707+0.924+1-0.924+0);
\draw[dashed](11.056-1-0.924-0.707-0.383-0-0.383-0.707,0+0.383+0.707+0.924+1-0.924+0)--(11.056-1-0.924-0.707-0.383-0-0.383-0.707,0+0.383+0.707+0.924+1-0.924-0.707);
\draw[dashed](11.056-1-0.924-0.707-0.383-0-0.383-0.707-0.942,0+0.383+0.707+0.924+1-0.924-0.707-0.383)--(11.056-1-0.924-0.707-0.383-0-0.383-0.707-0.942-0.5,0+0.383+0.707+0.924+1-0.924-0.707-0.383+0.5);
\draw[dashed](11.056-1-0.924-0.707-0.383-0-0.383-0.707-0.942-0.5,0+0.383+0.707+0.924+1-0.924-0.707-0.383+0.5)--(11.056-1-0.924-0.707-0.383-0-0.383-0.707-0.942-1,0+0.383+0.707+0.924+1-0.924-0.707-0.383);
\end{tikzpicture}
\caption{The open sets $V,S_0(V),\cdots,S_m(V)$ and geometrical relation for $m\equiv0$ mod $4$ where $m\ge4$.}\label{0}
\end{figure}

\begin{figure}[H]
\begin{tikzpicture}[scale = 15/11.836]
\draw(0.3,0) arc (0:42.4:0.3);
\draw(1.4,0) arc (0:21.2:0.4);
\draw(1+0.932+0.738+0.446+0.092-0.05,0+0.362+0.674+0.895+0.996-0.4) arc (-100:-75:0.4);
\draw node[below]at(0,0){$0$};
\draw node[below]at(11.836,0){$1$};
\draw node[font=\tiny]at(0.1,0.4){$\frac{(m-1)\pi}{4m}$};
\draw node[font=\tiny]at(1.5,-0.2){$\frac{2\pi}{m}$};
\draw node[font=\tiny]at(1+0.932+0.738+0.446+0.092+0.07,0+0.362+0.674+0.895+0.996-0.7){$\frac{2\pi}{m}$};
\draw[-](0,0)--(1,0);
\draw[-](1,0)--(1+0.932,0+0.362);
\draw node[outer sep=0pt,font=\tiny]at(1+0.932+0.1845,0+0.362+0.1685){.};
\draw node[outer sep=0pt,font=\tiny]at(1+0.932+0.1845+0.1845,0+0.362+0.1685+0.1685){.};
\draw node[outer sep=0pt,font=\tiny]at(1+0.932+0.1845+0.1845+0.1845,0+0.362+0.1685+0.1685+0.1685){.};
\draw[-](1+0.932+0.738,0+0.362+0.674)--(1+0.932+0.738+0.446,0+0.362+0.674+0.895);
\draw[-](1+0.932+0.738+0.446,0+0.362+0.674+0.895)--(1+0.932+0.738+0.446+0.092,0+0.362+0.674+0.895+0.996);
\draw[-](1+0.932+0.738+0.446+0.092,0+0.362+0.674+0.895+0.996)--(1+0.932+0.738+0.446+0.092+0.274,0+0.362+0.674+0.895+0.996-0.962);
\draw[-](1+0.932+0.738+0.446+0.092+0.274,0+0.362+0.674+0.895+0.996-0.962)--(1+0.932+0.738+0.446+0.092+0.274+0.603,0+0.362+0.674+0.895+0.996-0.962-0.798);
\draw node[outer sep=0pt,font=\tiny]at(1+0.932+0.738+0.446+0.092+0.274+0.603+0.2125,0+0.362+0.674+0.895+0.996-0.962-0.798-0.13175){.};
\draw node[outer sep=0pt,font=\tiny]at(1+0.932+0.738+0.446+0.092+0.274+0.603+0.2125+0.2125,0+0.362+0.674+0.895+0.996-0.962-0.798-0.13175-0.13175){.};
\draw node[outer sep=0pt,font=\tiny]at(1+0.932+0.738+0.446+0.092+0.274+0.603+0.2125+0.2125+0.2125,0+0.362+0.674+0.895+0.996-0.962-0.798-0.13175-0.13175-0.13175){.};
\draw[-](1+0.932+0.738+0.446+0.092+0.274+0.603+0.850,0+0.362+0.674+0.895+0.996-0.962-0.798-0.527)--(1+0.932+0.738+0.446+0.092+0.274+0.603+0.850+0.983,0+0.362+0.674+0.895+0.996-0.962-0.798-0.527-0.184);
\draw[-](1+0.932+0.738+0.446+0.092+0.274+0.603+0.850+0.983,0+0.362+0.674+0.895+0.996-0.962-0.798-0.527-0.184)--(1+0.932+0.738+0.446+0.092+0.274+0.603+0.850+0.983+0.983,0+0.362+0.674+0.895+0.996-0.962-0.798-0.527);
\draw node[outer sep=0pt,font=\tiny]at(1+0.932+0.738+0.446+0.092+0.274+0.603+0.850+0.983+0.983+0.2125,0+0.362+0.674+0.895+0.996-0.962-0.798-0.527+0.13175){.};
\draw node[outer sep=0pt,font=\tiny]at(1+0.932+0.738+0.446+0.092+0.274+0.603+0.850+0.983+0.983+0.2125+0.2125,0+0.362+0.674+0.895+0.996-0.962-0.798-0.527+0.13175+0.13175){.};
\draw node[outer sep=0pt,font=\tiny]at(1+0.932+0.738+0.446+0.092+0.274+0.603+0.850+0.983+0.983+0.2125+0.2125+0.2125,0+0.362+0.674+0.895+0.996-0.962-0.798-0.527+0.13175+0.13175+0.13175){.};
\draw[-](1+0.932+0.738+0.446+0.092+0.274+0.603+0.850+0.983+0.983+0.850,0+0.362+0.674+0.895+0.996-0.962-0.798)--(1+0.932+0.738+0.446+0.092+0.274+0.603+0.850+0.983+0.983+0.850+0.603,0+0.362+0.674+0.895+0.996-0.962);
\draw[-](1+0.932+0.738+0.446+0.092+0.274+0.603+0.850+0.983+0.983+0.850+0.603,0+0.362+0.674+0.895+0.996-0.962)--(1+0.932+0.738+0.446+0.092+0.274+0.603+0.850+0.983+0.983+0.850+0.603+0.274,0+0.362+0.674+0.895+0.996);
\draw[-](1+0.932+0.738+0.446+0.092+0.274+0.603+0.850+0.983+0.983+0.850+0.603+0.274,0+0.362+0.674+0.895+0.996)--(1+0.932+0.738+0.446+0.092+0.274+0.603+0.850+0.983+0.983+0.850+0.603+0.274+0.092,0+0.362+0.674+0.895);
\draw[-](1+0.932+0.738+0.446+0.092+0.274+0.603+0.850+0.983+0.983+0.850+0.603+0.274+0.092,0+0.362+0.674+0.895)--(1+0.932+0.738+0.446+0.092+0.274+0.603+0.850+0.983+0.983+0.850+0.603+0.274+0.092+0.446,0+0.362+0.674);
\draw node[outer sep=0pt,font=\tiny]at(1+0.932+0.738+0.446+0.092+0.274+0.603+0.850+0.983+0.983+0.850+0.603+0.274+0.092+0.446+0.1845,0+0.362+0.674-0.1685){.};
\draw node[outer sep=0pt,font=\tiny]at(1+0.932+0.738+0.446+0.092+0.274+0.603+0.850+0.983+0.983+0.850+0.603+0.274+0.092+0.446+0.1845+0.1845,0+0.362+0.674-0.1685-0.1685){.};
\draw node[outer sep=0pt,font=\tiny]at(1+0.932+0.738+0.446+0.092+0.274+0.603+0.850+0.983+0.983+0.850+0.603+0.274+0.092+0.446+0.1845+0.1845+0.1845,0+0.362+0.674-0.1685-0.1685-0.1685){.};
\draw[-](1+0.932+0.738+0.446+0.092+0.274+0.603+0.850+0.983+0.983+0.850+0.603+0.274+0.092+0.446+0.738,0+0.362)--(1+0.932+0.738+0.446+0.092+0.274+0.603+0.850+0.983+0.983+0.850+0.603+0.274+0.092+0.446+0.738+0.932,0);
\draw[-](1+0.932+0.738+0.446+0.092+0.274+0.603+0.850+0.983+0.983+0.850+0.603+0.274+0.092+0.446+0.738+0.932,0)--(1+0.932+0.738+0.446+0.092+0.274+0.603+0.850+0.983+0.983+0.850+0.603+0.274+0.092+0.446+0.738+0.932+1,0);
\draw[dashed](1,0)--(10.836,0);
\draw[dashed](0,0)--(5.918,5.404);
\draw[dashed](5.918,5.404)--(11.836,0);
\draw[dashed](0.5,0.457)--(1,0);
\draw[dashed](1,0)--(1+0.302,0+0.606);
\draw[dashed](1+0.302,0+0.606)--(1+0.932,0+0.362);
\draw[dashed](1+0.932+0.738,0+0.362+0.674)--(1+0.932+0.738-0.185,0+0.362+0.674+0.651);
\draw[dashed](1+0.932+0.738-0.185,0+0.362+0.674+0.651)--(1+0.932+0.738+0.446,0+0.362+0.674+0.895);
\draw[dashed](1+0.932+0.738+0.446,0+0.362+0.674+0.895)--(1+0.932+0.738+0.44-0.408,0+0.362+0.674+0.895+0.540);
\draw[dashed](1+0.932+0.738+0.446+0.092,0+0.362+0.674+0.895+0.996)--(1+0.932+0.738+0.446+0.092+0.575,0+0.362+0.674+0.895+0.996-0.357);
\draw[dashed](1+0.932+0.738+0.446+0.092+0.575,0+0.362+0.674+0.895+0.996-0.357)--(1+0.932+0.738+0.446+0.092+0.274,0+0.362+0.674+0.895+0.996-0.962);
\draw[dashed](1+0.932+0.738+0.446+0.092+0.274,0+0.362+0.674+0.895+0.996-0.962)--(1+0.932+0.738+0.446+0.092+0.274+0.665,0+0.362+0.674+0.895+0.996-0.962-0.125);
\draw[dashed](1+0.932+0.738+0.446+0.092+0.274+0.665,0+0.362+0.674+0.895+0.996-0.962-0.125)--(1+0.932+0.738+0.446+0.092+0.274+0.603,0+0.362+0.674+0.895+0.996-0.962-0.798);
\draw[dashed](1+0.932+0.738+0.446+0.092+0.274+0.603+0.850,0+0.362+0.674+0.895+0.996-0.962-0.798-0.527)--(1+0.932+0.738+0.446+0.092+0.274+0.603+0.850+0.575,0+0.362+0.674+0.895+0.996-0.962-0.798-0.527+0.357);
\draw[dashed](1+0.932+0.738+0.446+0.092+0.274+0.603+0.850+0.575,0+0.362+0.674+0.895+0.996-0.962-0.798-0.527+0.357)--(1+0.932+0.738+0.446+0.092+0.274+0.603+0.850+0.983,0+0.362+0.674+0.895+0.996-0.962-0.798-0.527-0.184);
\draw[dashed](11.836-0.5,0.457)--(11.836-1,0);
\draw[dashed](11.836-1,0)--(11.836-1-0.302,0+0.606);
\draw[dashed](11.836-1-0.302,0+0.606)--(11.836-1-0.932,0+0.362);
\draw[dashed](11.836-1-0.932-0.738,0+0.362+0.674)--(11.836-1-0.932-0.738+0.185,0+0.362+0.674+0.651);
\draw[dashed](11.836-1-0.932-0.738+0.185,0+0.362+0.674+0.651)--(11.836-1-0.932-0.738-0.446,0+0.362+0.674+0.895);
\draw[dashed](11.836-1-0.932-0.738-0.446,0+0.362+0.674+0.895)--(11.836-1-0.932-0.738-0.44+0.408,0+0.362+0.674+0.895+0.540);
\draw[dashed](11.836-1-0.932-0.738-0.446-0.092,0+0.362+0.674+0.895+0.996)--(11.836-1-0.932-0.738-0.446-0.092-0.575,0+0.362+0.674+0.895+0.996-0.357);
\draw[dashed](11.836-1-0.932-0.738-0.446-0.092-0.575,0+0.362+0.674+0.895+0.996-0.357)--(11.836-1-0.932-0.738-0.446-0.092-0.274,0+0.362+0.674+0.895+0.996-0.962);
\draw[dashed](11.836-1-0.932-0.738-0.446-0.092-0.274,0+0.362+0.674+0.895+0.996-0.962)--(11.836-1-0.932-0.738-0.446-0.092-0.274-0.665,0+0.362+0.674+0.895+0.996-0.962-0.125);
\draw[dashed](11.836-1-0.932-0.738-0.446-0.092-0.274-0.665,0+0.362+0.674+0.895+0.996-0.962-0.125)--(11.836-1-0.932-0.738-0.446-0.092-0.274-0.603,0+0.362+0.674+0.895+0.996-0.962-0.798);
\draw[dashed](11.836-1-0.932-0.738-0.446-0.092-0.274-0.603-0.850,0+0.362+0.674+0.895+0.996-0.962-0.798-0.527)--(11.836-1-0.932-0.738-0.446-0.092-0.274-0.603-0.850-0.575,0+0.362+0.674+0.895+0.996-0.962-0.798-0.527+0.357);
\draw[dashed](11.836-1-0.932-0.738-0.446-0.092-0.274-0.603-0.850-0.575,0+0.362+0.674+0.895+0.996-0.962-0.798-0.527+0.357)--(11.836-1-0.932-0.738-0.446-0.092-0.274-0.603-0.850-0.983,0+0.362+0.674+0.895+0.996-0.962-0.798-0.527-0.184);
\end{tikzpicture}
\caption{The open sets $V,S_0(V),\cdots,S_m(V)$ and geometrical relation for $m\equiv1$ mod $4$ where $m\ge4$.}\label{1}
\end{figure}

\begin{figure}[H]
\begin{tikzpicture}[scale = 15/12.52]
\draw(0.3,0) arc (0:40:0.3);
\draw(1.4,0) arc (0:20:0.4);
\draw(1+0.940+0.766+0.5+0.174-0.07,0+0.342+0.643+0.866+0.985-0.4) arc (-90:-70:0.4);
\draw node[below]at(0,0){$0$};
\draw node[below]at(12.52,0){$1$};
\draw node[font=\tiny]at(0.1,0.4){$\frac{(m-2)\pi}{4m}$};
\draw node[font=\tiny]at(1.5,-0.2){$\frac{2\pi}{m}$};
\draw node[font=\tiny]at(1+0.940+0.766+0.5+0.174-0,0+0.342+0.643+0.866+0.985-0.8){$\frac{2\pi}{m}$};
\draw[-](0,0)--(1,0);
\draw[-](1,0)--(1+0.940,0+0.342);
\draw node[outer sep=0pt,font=\tiny]at(1+0.940+0.1915,0+0.342+0.16075){.};
\draw node[outer sep=0pt,font=\tiny]at(1+0.940+0.1915+0.1915,0+0.342+0.16075+0.16075){.};
\draw node[outer sep=0pt,font=\tiny]at(1+0.940+0.1915+0.1915+0.1915,0+0.342+0.16075+0.16075+0.16075){.};
\draw[-](1+0.940+0.766,0+0.342+0.643)--(1+0.940+0.766+0.5,0+0.342+0.643+0.866);
\draw[-](1+0.940+0.766+0.5,0+0.342+0.643+0.866)--(1+0.940+0.766+0.5+0.174,0+0.342+0.643+0.866+0.985);
\draw[-](1+0.940+0.766+0.5+0.174,0+0.342+0.643+0.866+0.985)--(1+0.940+0.766+0.5+0.174+0.174,0+0.342+0.643+0.866);
\draw[-](1+0.940+0.766+0.5+0.174+0.174,0+0.342+0.643+0.866)--(1+0.940+0.766+0.5+0.174+0.174+0.5,0+0.342+0.643);
\draw node[outer sep=0pt,font=\tiny]at(1+0.940+0.766+0.5+0.174+0.174+0.5+0.1915,0+0.342+0.643-0.16075){.};
\draw node[outer sep=0pt,font=\tiny]at(1+0.940+0.766+0.5+0.174+0.174+0.5+0.1915+0.1915,0+0.342+0.643-0.16075-0.16075){.};
\draw node[outer sep=0pt,font=\tiny]at(1+0.940+0.766+0.5+0.174+0.174+0.5+0.1915+0.1915+0.1915,0+0.342+0.643-0.16075-0.16075-0.16075){.};
\draw[-](1+0.940+0.766+0.5+0.174+0.174+0.5+0.766,0+0.342)--(1+0.940+0.766+0.5+0.174+0.174+0.5+0.766+0.940,0);
\draw[-](1+0.940+0.766+0.5+0.174+0.174+0.5+0.766+0.940,0)--(1+0.940+0.766+0.5+0.174+0.174+0.5+0.766+0.940+1,0);
\draw[-](1+0.940+0.766+0.5+0.174+0.174+0.5+0.766+0.940+1,0)--(1+0.940+0.766+0.5+0.174+0.174+0.5+0.766+0.940+1+0.940,0+0.342);
\draw node[outer sep=0pt,font=\tiny]at(1+0.940+0.766+0.5+0.174+0.174+0.5+0.766+0.940+1+0.940+0.1915,0+0.342+0.16075){.};
\draw node[outer sep=0pt,font=\tiny]at(1+0.940+0.766+0.5+0.174+0.174+0.5+0.766+0.940+1+0.940+0.1915+0.1915,0+0.342+0.16075+0.16075){.};
\draw node[outer sep=0pt,font=\tiny]at(1+0.940+0.766+0.5+0.174+0.174+0.5+0.766+0.940+1+0.940+0.1915+0.1915+0.1915,0+0.342+0.16075+0.16075+0.16075){.};
\draw[-](1+0.940+0.766+0.5+0.174+0.174+0.5+0.766+0.940+1+0.940+0.766,0+0.342+0.643)--(1+0.940+0.766+0.5+0.174+0.174+0.5+0.766+0.940+1+0.940+0.766+0.5,0+0.342+0.643+0.866);
\draw[-](1+0.940+0.766+0.5+0.174+0.174+0.5+0.766+0.940+1+0.940+0.766+0.5,0+0.342+0.643+0.866)--(1+0.940+0.766+0.5+0.174+0.174+0.5+0.766+0.940+1+0.940+0.766+0.5+0.174,0+0.342+0.643+0.866+0.985);
\draw[-](1+0.940+0.766+0.5+0.174+0.174+0.5+0.766+0.940+1+0.940+0.766+0.5+0.174,0+0.342+0.643+0.866+0.985)--(1+0.940+0.766+0.5+0.174+0.174+0.5+0.766+0.940+1+0.940+0.766+0.5+0.174+0.174,0+0.342+0.643+0.866);
\draw[-](1+0.940+0.766+0.5+0.174+0.174+0.5+0.766+0.940+1+0.940+0.766+0.5+0.174+0.174,0+0.342+0.643+0.866)--(1+0.940+0.766+0.5+0.174+0.174+0.5+0.766+0.940+1+0.940+0.766+0.5+0.174+0.174+0.5,0+0.342+0.643);
\draw node[outer sep=0pt,font=\tiny]at(1+0.940+0.766+0.5+0.174+0.174+0.5+0.766+0.940+1+0.940+0.766+0.5+0.174+0.174+0.5+0.1915,0+0.342+0.643-0.16075){.};
\draw node[outer sep=0pt,font=\tiny]at(1+0.940+0.766+0.5+0.174+0.174+0.5+0.766+0.940+1+0.940+0.766+0.5+0.174+0.174+0.5+0.1915+0.1915,0+0.342+0.643-0.16075-0.16075){.};
\draw node[outer sep=0pt,font=\tiny]at(1+0.940+0.766+0.5+0.174+0.174+0.5+0.766+0.940+1+0.940+0.766+0.5+0.174+0.174+0.5+0.1915+0.1915+0.1915,0+0.342+0.643-0.16075-0.16075-0.16075){.};
\draw[-](1+0.940+0.766+0.5+0.174+0.174+0.5+0.766+0.940+1+0.940+0.766+0.5+0.174+0.174+0.5+0.766,0+0.342)--(1+0.940+0.766+0.5+0.174+0.174+0.5+0.766+0.940+1+0.940+0.766+0.5+0.174+0.174+0.5+0.766+0.940,0);
\draw[-](1+0.940+0.766+0.5+0.174+0.174+0.5+0.766+0.940+1+0.940+0.766+0.5+0.174+0.174+0.5+0.766+0.940,0)--(1+0.940+0.766+0.5+0.174+0.174+0.5+0.766+0.940+1+0.940+0.766+0.5+0.174+0.174+0.5+0.766+0.940+1,0);
\draw[dashed](0.5,0.420)--(1,0);
\draw[dashed](1,0)--(1+0.3265,0.5655);
\draw[dashed](1+0.3265,0.5655)--(1+0.940,0+0.342);
\draw[dashed](1+0.940+0.766,0+0.342+0.643)--(1+0.940+0.766-0.113,0+0.342+0.643+0.643);
\draw[dashed](1+0.940+0.766-0.113,0+0.342+0.643+0.643)--(1+0.940+0.766+0.5,0+0.342+0.643+0.866);
\draw[dashed](1+0.940+0.766+0.5,0+0.342+0.643+0.866)--(1+0.940+0.766+0.5-0.3265,0+0.342+0.643+0.866+0.5655);
\draw[dashed](1+0.940+0.766+0.5+0.174,0+0.342+0.643+0.866+0.985)--(1+0.940+0.766+0.5+0.174+0.5,0+0.342+0.643+0.866+0.985-0.420);
\draw[dashed](1+0.940+0.766+0.5+0.174+0.5,0+0.342+0.643+0.866+0.985-0.420)--(1+0.940+0.766+0.5+0.174+0.174,0+0.342+0.643+0.866);
\draw[dashed](1+0.940+0.766+0.5+0.174+0.174,0+0.342+0.643+0.866)--(1+0.940+0.766+0.5+0.174+0.174+0.614,0+0.342+0.643+0.866-0.223);
\draw[dashed](1+0.940+0.766+0.5+0.174+0.174+0.614,0+0.342+0.643+0.866-0.223)--(1+0.940+0.766+0.5+0.174+0.174+0.5,0+0.342+0.643);
\draw[dashed](1+0.940+0.766+0.5+0.174+0.174+0.5+0.766,0+0.342)--(1+0.940+0.766+0.5+0.174+0.174+0.5+0.766+0.614,0+0.342+0.223);
\draw[dashed](1+0.940+0.766+0.5+0.174+0.174+0.5+0.766+0.614,0+0.342+0.223)--(1+0.940+0.766+0.5+0.174+0.174+0.5+0.766+0.940,0);
\draw[dashed](12.52-0.5,0.420)--(12.52-1,0);
\draw[dashed](12.52-1,0)--(12.52-1-0.3265,0.5655);
\draw[dashed](12.52-1-0.3265,0.5655)--(12.52-1-0.940,0+0.342);
\draw[dashed](12.52-1-0.940-0.766,0+0.342+0.643)--(12.52-1-0.940-0.766+0.113,0+0.342+0.643+0.643);
\draw[dashed](12.52-1-0.940-0.766+0.113,0+0.342+0.643+0.643)--(12.52-1-0.940-0.766-0.5,0+0.342+0.643+0.866);
\draw[dashed](12.52-1-0.940-0.766-0.5,0+0.342+0.643+0.866)--(12.52-1-0.940-0.766-0.5+0.3265,0+0.342+0.643+0.866+0.5655);
\draw[dashed](12.52-1-0.940-0.766-0.5-0.174,0+0.342+0.643+0.866+0.985)--(12.52-1-0.940-0.766-0.5-0.174-0.5,0+0.342+0.643+0.866+0.985-0.420);
\draw[dashed](12.52-1-0.940-0.766-0.5-0.174-0.5,0+0.342+0.643+0.866+0.985-0.420)--(12.52-1-0.940-0.766-0.5-0.174-0.174,0+0.342+0.643+0.866);
\draw[dashed](12.52-1-0.940-0.766-0.5-0.174-0.174,0+0.342+0.643+0.866)--(12.52-1-0.940-0.766-0.5-0.174-0.174-0.614,0+0.342+0.643+0.866-0.223);
\draw[dashed](12.52-1-0.940-0.766-0.5-0.174-0.174-0.614,0+0.342+0.643+0.866-0.223)--(12.52-1-0.940-0.766-0.5-0.174-0.174-0.5,0+0.342+0.643);
\draw[dashed](12.52-1-0.940-0.766-0.5-0.174-0.174-0.5-0.766,0+0.342)--(12.52-1-0.940-0.766-0.5-0.174-0.174-0.5-0.766-0.614,0+0.342+0.223);
\draw[dashed](12.52-1-0.940-0.766-0.5-0.174-0.174-0.5-0.766-0.614,0+0.342+0.223)--(12.52-1-0.940-0.766-0.5-0.174-0.174-0.5-0.766-0.940,0);
\draw[dashed](12.52-1-0.940-0.766-0.5-0.174-0.174-0.5-0.766-0.940,0)--(12.52-1-0.940-0.766-0.5-0.174-0.174-0.5-0.766-0.940-0.5,0+0.420);
\draw[dashed](12.52-1-0.940-0.766-0.5-0.174-0.174-0.5-0.766-0.940-0.5,0+0.420)--(12.52-1-0.940-0.766-0.5-0.174-0.174-0.5-0.766-0.940-1,0);
\draw[dashed](1,0)--(5.76,0);
\draw[dashed](6.76,0)--(11.52,0);
\draw[dashed](0,0)--(6.26,5.2527637);
\draw[dashed](6.26,5.2527637)--(12.52,0);
\end{tikzpicture}
\caption{The open sets $V,S_0(V),\cdots,S_m(V)$ and geometrical relation for $m\equiv2$ mod $4$ where $m\ge4$.}\label{2}
\end{figure}

\begin{figure}[H]
\begin{tikzpicture}[scale = 15/13.11]
\draw(0.3,0) arc (0:37.9:0.3);
\draw(1.8,0) arc (0:18.95:0.8);
\draw(1.5,0) arc (0:-4.737:1.5);
\draw(1+0.946+0.789+0.548+0.245-0.18,0+0.324+0.614+0.837+0.969-0.7) arc (-98:-78:0.7);
\draw node[below]at(0,0){$0$};
\draw node[below]at(13.11,0){$1$};
\draw node[font=\tiny]at(0.1,0.4){$\frac{(m-3)\pi}{4m}$};
\draw node[font=\tiny]at(2.15,0.2){$\frac{2\pi}{m}$};
\draw node[font=\tiny]at(1.65,-0.35){$\frac{\pi}{2m}$};
\draw node[font=\tiny]at(1+0.946+0.789+0.548+0.245-0.075,0+0.324+0.614+0.837+0.969-1){$\frac{2\pi}{m}$};
\draw[dashed](1,0)--(12.11,0);
\draw[-](0,0)--(1,0);
\draw[-](1,0)--(1+0.946,0+0.324);
\draw node[outer sep=0pt,font=\tiny]at(1+0.946+0.19725,0+0.324+0.1535){.};
\draw node[outer sep=0pt,font=\tiny]at(1+0.946+0.19725+0.19725,0+0.324+0.1535+0.1535){.};
\draw node[outer sep=0pt,font=\tiny]at(1+0.946+0.19725+0.19725+0.19725,0+0.324+0.1535+0.1535+0.1535){.};
\draw[-](1+0.946+0.789,0+0.324+0.614)--(1+0.946+0.789+0.548,0+0.324+0.614+0.837);
\draw[-](1+0.946+0.789+0.548,0+0.324+0.614+0.837)--(1+0.946+0.789+0.548+0.245,0+0.324+0.614+0.837+0.969);
\draw[-](1+0.946+0.789+0.548+0.245,0+0.324+0.614+0.837+0.969)--(1+0.946+0.789+0.548+0.245+0.082,0+0.324+0.614+0.837+0.969-0.997);
\draw[-](1+0.946+0.789+0.548+0.245+0.082,0+0.324+0.614+0.837+0.969-0.997)--(1+0.946+0.789+0.548+0.245+0.082+0.402,0+0.324+0.614+0.837+0.969-0.997-0.916);
\draw node[outer sep=0pt,font=\tiny]at(1+0.946+0.789+0.548+0.245+0.082+0.402+0.16925,0+0.324+0.614+0.837+0.969-0.997-0.916-0.184){.};
\draw node[outer sep=0pt,font=\tiny]at(1+0.946+0.789+0.548+0.245+0.082+0.402+0.16925+0.16925,0+0.324+0.614+0.837+0.969-0.997-0.916-0.184-0.184){.};
\draw node[outer sep=0pt,font=\tiny]at(1+0.946+0.789+0.548+0.245+0.082+0.402+0.16925+0.16925+0.16925,0+0.324+0.614+0.837+0.969-0.997-0.916-0.184-0.184-0.184){.};
\draw[-](1+0.946+0.789+0.548+0.245+0.082+0.402+0.677,0+0.324+0.614+0.837+0.969-0.997-0.916-0.736)--(1+0.946+0.789+0.548+0.245+0.082+0.402+0.677+0.880,0+0.324+0.614+0.837+0.969-0.997-0.916-0.736-0.476);
\draw[-](1+0.946+0.789+0.548+0.245+0.082+0.402+0.677+0.880,0+0.324+0.614+0.837+0.969-0.997-0.916-0.736-0.476)--(1+0.946+0.789+0.548+0.245+0.082+0.402+0.677+0.880+0.986,0+0.324+0.614+0.837+0.969-0.997-0.916-0.736-0.476-0.165);
\draw[-](1+0.946+0.789+0.548+0.245+0.082+0.402+0.677+0.880+0.986,0+0.324+0.614+0.837+0.969-0.997-0.916-0.736-0.476-0.165)--(1+0.946+0.789+0.548+0.245+0.082+0.402+0.677+0.880+0.986+0.986,0+0.324+0.614+0.837+0.969-0.997-0.916-0.736-0.476);
\draw[-](1+0.946+0.789+0.548+0.245+0.082+0.402+0.677+0.880+0.986+0.986,0+0.324+0.614+0.837+0.969-0.997-0.916-0.736-0.476)--(1+0.946+0.789+0.548+0.245+0.082+0.402+0.677+0.880+0.986+0.986+0.880,0+0.324+0.614+0.837+0.969-0.997-0.916-0.736);
\draw node[outer sep=0pt,font=\tiny]at(1+0.946+0.789+0.548+0.245+0.082+0.402+0.677+0.880+0.986+0.986+0.880+0.16925,0+0.324+0.614+0.837+0.969-0.997-0.916-0.736+0.184){.};
\draw node[outer sep=0pt,font=\tiny]at(1+0.946+0.789+0.548+0.245+0.082+0.402+0.677+0.880+0.986+0.986+0.880+0.16925+0.16925,0+0.324+0.614+0.837+0.969-0.997-0.916-0.736+0.184+0.184){.};
\draw node[outer sep=0pt,font=\tiny]at(1+0.946+0.789+0.548+0.245+0.082+0.402+0.677+0.880+0.986+0.986+0.880+0.16925+0.16925+0.16925,0+0.324+0.614+0.837+0.969-0.997-0.916-0.736+0.184+0.184+0.184){.};
\draw[-](1+0.946+0.789+0.548+0.245+0.082+0.402+0.677+0.880+0.986+0.986+0.880+0.677,0+0.324+0.614+0.837+0.969-0.997-0.916)--(1+0.946+0.789+0.548+0.245+0.082+0.402+0.677+0.880+0.986+0.986+0.880+0.677+0.402,0+0.324+0.614+0.837+0.969-0.997);
\draw[-](1+0.946+0.789+0.548+0.245+0.082+0.402+0.677+0.880+0.986+0.986+0.880+0.677+0.402,0+0.324+0.614+0.837+0.969-0.997)--(1+0.946+0.789+0.548+0.245+0.082+0.402+0.677+0.880+0.986+0.986+0.880+0.677+0.402+0.082,0+0.324+0.614+0.837+0.969);
\draw[-](1+0.946+0.789+0.548+0.245+0.082+0.402+0.677+0.880+0.986+0.986+0.880+0.677+0.402+0.082,0+0.324+0.614+0.837+0.969)--(1+0.946+0.789+0.548+0.245+0.082+0.402+0.677+0.880+0.986+0.986+0.880+0.677+0.402+0.082+0.245,0+0.324+0.614+0.837);
\draw[-](1+0.946+0.789+0.548+0.245+0.082+0.402+0.677+0.880+0.986+0.986+0.880+0.677+0.402+0.082+0.245,0+0.324+0.614+0.837)--(1+0.946+0.789+0.548+0.245+0.082+0.402+0.677+0.880+0.986+0.986+0.880+0.677+0.402+0.082+0.245+0.548,0+0.324+0.614);
\draw node[outer sep=0pt,font=\tiny]at(1+0.946+0.789+0.548+0.245+0.082+0.402+0.677+0.880+0.986+0.986+0.880+0.677+0.402+0.082+0.245+0.548+0.19725,0+0.324+0.614-0.1535){.};
\draw node[outer sep=0pt,font=\tiny]at(1+0.946+0.789+0.548+0.245+0.082+0.402+0.677+0.880+0.986+0.986+0.880+0.677+0.402+0.082+0.245+0.548+0.19725+0.19725,0+0.324+0.614-0.1535-0.1535){.};
\draw node[outer sep=0pt,font=\tiny]at(1+0.946+0.789+0.548+0.245+0.082+0.402+0.677+0.880+0.986+0.986+0.880+0.677+0.402+0.082+0.245+0.548+0.19725+0.19725+0.19725,0+0.324+0.614-0.1535-0.1535-0.1535){.};
\draw[-](1+0.946+0.789+0.548+0.245+0.082+0.402+0.677+0.880+0.986+0.986+0.880+0.677+0.402+0.082+0.245+0.548+0.789,0+0.324)--(1+0.946+0.789+0.548+0.245+0.082+0.402+0.677+0.880+0.986+0.986+0.880+0.677+0.402+0.082+0.245+0.548+0.789+0.946,0);
\draw[-](1+0.946+0.789+0.548+0.245+0.082+0.402+0.677+0.880+0.986+0.986+0.880+0.677+0.402+0.082+0.245+0.548+0.789+0.946,0)--(1+0.946+0.789+0.548+0.245+0.082+0.402+0.677+0.880+0.986+0.986+0.880+0.677+0.402+0.082+0.245+0.548+0.789+0.946+1,0);
\draw[dashed](0.5,0.389)--(1,0);
\draw[dashed](1,0)--(1+0.34678,0+0.531);
\draw[dashed](1+0.34678,0+0.531)--(1+0.946,0+0.324);
\draw[dashed](1+0.946+0.789,0+0.324+0.614)--(1+0.946+0.789-0.052,0+0.324+0.614+0.632);
\draw[dashed](1+0.946+0.789-0.052,0+0.324+0.614+0.632)--(1+0.946+0.789+0.548,0+0.324+0.614+0.837);
\draw[dashed](1+0.946+0.789+0.548,0+0.324+0.614+0.837)--(1+0.946+0.789+0.548-0.255,0+0.324+0.614+0.837+0.581);
\draw[dashed](1+0.946+0.789+0.548+0.245,0+0.324+0.614+0.837+0.969)--(1+0.946+0.789+0.548+0.245+0.429,0+0.324+0.614+0.837+0.969-0.467);
\draw[dashed](1+0.946+0.789+0.548+0.245+0.429,0+0.324+0.614+0.837+0.969-0.467)--(1+0.946+0.789+0.548+0.245+0.082,0+0.324+0.614+0.837+0.969-0.997);
\draw[dashed](1+0.946+0.789+0.548+0.245+0.082,0+0.324+0.614+0.837+0.969-0.997)--(1+0.946+0.789+0.548+0.245+0.082+0.558,0+0.324+0.614+0.837+0.969-0.997-0.302);
\draw[dashed](1+0.946+0.789+0.548+0.245+0.082+0.558,0+0.324+0.614+0.837+0.969-0.997-0.302)--(1+0.946+0.789+0.548+0.245+0.082+0.402,0+0.324+0.614+0.837+0.969-0.997-0.916);
\draw[dashed](1+0.946+0.789+0.548+0.245+0.082+0.402+0.677,0+0.324+0.614+0.837+0.969-0.997-0.916-0.736)--(1+0.946+0.789+0.548+0.245+0.082+0.402+0.677+0.625,0+0.324+0.614+0.837+0.969-0.997-0.916-0.736+0.105);
\draw[dashed](1+0.946+0.789+0.548+0.245+0.082+0.402+0.677+0.625,0+0.324+0.614+0.837+0.969-0.997-0.916-0.736+0.105)--(1+0.946+0.789+0.548+0.245+0.082+0.402+0.677+0.880,0+0.324+0.614+0.837+0.969-0.997-0.916-0.736-0.476);
\draw[dashed](1+0.946+0.789+0.548+0.245+0.082+0.402+0.677+0.880,0+0.324+0.614+0.837+0.969-0.997-0.916-0.736-0.476)--(1+0.946+0.789+0.548+0.245+0.082+0.402+0.677+0.880+0.558,0+0.324+0.614+0.837+0.969-0.997-0.916-0.736-0.476+0.302);
\draw[dashed](1+0.946+0.789+0.548+0.245+0.082+0.402+0.677+0.880+0.558,0+0.324+0.614+0.837+0.969-0.997-0.916-0.736-0.476+0.302)--(1+0.946+0.789+0.548+0.245+0.082+0.402+0.677+0.880+0.986,0+0.324+0.614+0.837+0.969-0.997-0.916-0.736-0.476-0.165);
\draw[dashed](13.11-0.5,0.389)--(13.11-1,0);
\draw[dashed](13.11-1,0)--(13.11-1-0.34678,0+0.531);
\draw[dashed](13.11-1-0.34678,0+0.531)--(13.11-1-0.946,0+0.324);
\draw[dashed](13.11-1-0.946-0.789,0+0.324+0.614)--(13.11-1-0.946-0.789+0.052,0+0.324+0.614+0.632);
\draw[dashed](13.11-1-0.946-0.789+0.052,0+0.324+0.614+0.632)--(13.11-1-0.946-0.789-0.548,0+0.324+0.614+0.837);
\draw[dashed](13.11-1-0.946-0.789-0.548,0+0.324+0.614+0.837)--(13.11-1-0.946-0.789-0.548+0.255,0+0.324+0.614+0.837+0.581);
\draw[dashed](13.11-1-0.946-0.789-0.548-0.245,0+0.324+0.614+0.837+0.969)--(13.11-1-0.946-0.789-0.548-0.245-0.429,0+0.324+0.614+0.837+0.969-0.467);
\draw[dashed](13.11-1-0.946-0.789-0.548-0.245-0.429,0+0.324+0.614+0.837+0.969-0.467)--(13.11-1-0.946-0.789-0.548-0.245-0.082,0+0.324+0.614+0.837+0.969-0.997);
\draw[dashed](13.11-1-0.946-0.789-0.548-0.245-0.082,0+0.324+0.614+0.837+0.969-0.997)--(13.11-1-0.946-0.789-0.548-0.245-0.082-0.558,0+0.324+0.614+0.837+0.969-0.997-0.302);
\draw[dashed](13.11-1-0.946-0.789-0.548-0.245-0.082-0.558,0+0.324+0.614+0.837+0.969-0.997-0.302)--(13.11-1-0.946-0.789-0.548-0.245-0.082-0.402,0+0.324+0.614+0.837+0.969-0.997-0.916);
\draw[dashed](13.11-1-0.946-0.789-0.548-0.245-0.082-0.402-0.677,0+0.324+0.614+0.837+0.969-0.997-0.916-0.736)--(13.11-1-0.946-0.789-0.548-0.245-0.082-0.402-0.677-0.625,0+0.324+0.614+0.837+0.969-0.997-0.916-0.736+0.105);
\draw[dashed](13.11-1-0.946-0.789-0.548-0.245-0.082-0.402-0.677-0.625,0+0.324+0.614+0.837+0.969-0.997-0.916-0.736+0.105)--(13.11-1-0.946-0.789-0.548-0.245-0.082-0.402-0.677-0.880,0+0.324+0.614+0.837+0.969-0.997-0.916-0.736-0.476);
\draw[dashed](13.11-1-0.946-0.789-0.548-0.245-0.082-0.402-0.677-0.880,0+0.324+0.614+0.837+0.969-0.997-0.916-0.736-0.476)--(13.11-1-0.946-0.789-0.548-0.245-0.082-0.402-0.677-0.880-0.558,0+0.324+0.614+0.837+0.969-0.997-0.916-0.736-0.476+0.302);
\draw[dashed](13.11-1-0.946-0.789-0.548-0.245-0.082-0.402-0.677-0.880-0.558,0+0.324+0.614+0.837+0.969-0.997-0.916-0.736-0.476+0.302)--(13.11-1-0.946-0.789-0.548-0.245-0.082-0.402-0.677-0.880-0.986,0+0.324+0.614+0.837+0.969-0.997-0.916-0.736-0.476-0.165);
\draw[dashed](0,0)--(6.555,5.103);
\draw[dashed](6.555,5.103)--(13.11,0);
\draw[dashed](0,0)--(1+0.946+0.789+0.548+0.245+0.082+0.402+0.677+0.880+0.986,0+0.324+0.614+0.837+0.969-0.997-0.916-0.736-0.476-0.165);
\draw[dashed](1+0.946+0.789+0.548+0.245+0.082+0.402+0.677+0.880+0.986,0+0.324+0.614+0.837+0.969-0.997-0.916-0.736-0.476-0.165)--(13.11,0);
\draw[dashed](0.5,-0.04143)--(1,0);
\draw[dashed](1,0)--(1+0.486,0+0.123);
\draw[dashed](1+0.486,0+0.123)--(1+0.946,0+0.324);
\draw[dashed](1+0.946+0.789,0+0.324+0.614)--(1+0.946+0.789+0.308,0+0.324+0.614+0.396);
\draw[dashed](1+0.946+0.789+0.308,0+0.324+0.614+0.396)--(1+0.946+0.789+0.548,0+0.324+0.614+0.837);
\draw[dashed](1+0.946+0.789+0.548,0+0.324+0.614+0.837)--(1+0.946+0.789+0.548+0.163,0+0.324+0.614+0.837+0.4745);
\draw[dashed](1+0.946+0.789+0.548+0.163,0+0.324+0.614+0.837+0.4745)--(1+0.946+0.789+0.548+0.245,0+0.324+0.614+0.837+0.969);
\draw[dashed](1+0.946+0.789+0.548+0.245,0+0.324+0.614+0.837+0.969)--(1+0.946+0.789+0.548+0.245,0+0.324+0.614+0.837+0.969-0.5017);
\draw[dashed](1+0.946+0.789+0.548+0.245,0+0.324+0.614+0.837+0.969-0.5017)--(1+0.946+0.789+0.548+0.245+0.082,0+0.324+0.614+0.837+0.969-0.997);
\draw[dashed](1+0.946+0.789+0.548+0.245+0.082,0+0.324+0.614+0.837+0.969-0.997)--(1+0.946+0.789+0.548+0.245+0.082+0.1625,0+0.324+0.614+0.837+0.969-0.997-0.474651);
\draw[dashed](1+0.946+0.789+0.548+0.245+0.082+0.1625,0+0.324+0.614+0.837+0.969-0.997-0.474651)--(1+0.946+0.789+0.548+0.245+0.082+0.402,0+0.324+0.614+0.837+0.969-0.997-0.916);
\draw[dashed](1+0.946+0.789+0.548+0.245+0.082+0.402+0.677,0+0.324+0.614+0.837+0.969-0.997-0.916-0.736)--(1+0.946+0.789+0.548+0.245+0.082+0.402+0.677++0.42,0+0.324+0.614+0.837+0.969-0.997-0.916-0.736-0.2744);
\draw[dashed](1+0.946+0.789+0.548+0.245+0.082+0.402+0.677++0.42,0+0.324+0.614+0.837+0.969-0.997-0.916-0.736-0.2744)--(1+0.946+0.789+0.548+0.245+0.082+0.402+0.677+0.880,0+0.324+0.614+0.837+0.969-0.997-0.916-0.736-0.476);
\draw[dashed](1+0.946+0.789+0.548+0.245+0.082+0.402+0.677+0.880,0+0.324+0.614+0.837+0.969-0.997-0.916-0.736-0.476)--(1+0.946+0.789+0.548+0.245+0.082+0.402+0.677+0.880+0.486,0+0.324+0.614+0.837+0.969-0.997-0.916-0.736-0.476-0.123);
\draw[dashed](13.11-0.5,-0.04143)--(13.11-1,0);
\draw[dashed](13.11-1,0)--(13.11-1-0.486,0+0.123);
\draw[dashed](13.11-1-0.486,0+0.123)--(13.11-1-0.946,0+0.324);
\draw[dashed](13.11-1-0.946-0.789,0+0.324+0.614)--(13.11-1-0.946-0.789-0.308,0+0.324+0.614+0.396);
\draw[dashed](13.11-1-0.946-0.789-0.308,0+0.324+0.614+0.396)--(13.11-1-0.946-0.789-0.548,0+0.324+0.614+0.837);
\draw[dashed](13.11-1-0.946-0.789-0.548,0+0.324+0.614+0.837)--(13.11-1-0.946-0.789-0.548-0.163,0+0.324+0.614+0.837+0.4745);
\draw[dashed](13.11-1-0.946-0.789-0.548-0.163,0+0.324+0.614+0.837+0.4745)--(13.11-1-0.946-0.789-0.548-0.245,0+0.324+0.614+0.837+0.969);
\draw[dashed](13.11-1-0.946-0.789-0.548-0.245,0+0.324+0.614+0.837+0.969)--(13.11-1-0.946-0.789-0.548-0.245,0+0.324+0.614+0.837+0.969-0.5017);
\draw[dashed](13.11-1-0.946-0.789-0.548-0.245,0+0.324+0.614+0.837+0.969-0.5017)--(13.11-1-0.946-0.789-0.548-0.245-0.082,0+0.324+0.614+0.837+0.969-0.997);
\draw[dashed](13.11-1-0.946-0.789-0.548-0.245-0.082,0+0.324+0.614+0.837+0.969-0.997)--(13.11-1-0.946-0.789-0.548-0.245-0.082-0.1625,0+0.324+0.614+0.837+0.969-0.997-0.474651);
\draw[dashed](13.11-1-0.946-0.789-0.548-0.245-0.082-0.1625,0+0.324+0.614+0.837+0.969-0.997-0.474651)--(13.11-1-0.946-0.789-0.548-0.245-0.082-0.402,0+0.324+0.614+0.837+0.969-0.997-0.916);
\draw[dashed](13.11-1-0.946-0.789-0.548-0.245-0.082-0.402-0.677,0+0.324+0.614+0.837+0.969-0.997-0.916-0.736)--(13.11-1-0.946-0.789-0.548-0.245-0.082-0.402-0.677-0.42,0+0.324+0.614+0.837+0.969-0.997-0.916-0.736-0.2744);
\draw[dashed](13.11-1-0.946-0.789-0.548-0.245-0.082-0.402-0.677-0.42,0+0.324+0.614+0.837+0.969-0.997-0.916-0.736-0.2744)--(13.11-1-0.946-0.789-0.548-0.245-0.082-0.402-0.677-0.880,0+0.324+0.614+0.837+0.969-0.997-0.916-0.736-0.476);
\draw[dashed](13.11-1-0.946-0.789-0.548-0.245-0.082-0.402-0.677-0.880,0+0.324+0.614+0.837+0.969-0.997-0.916-0.736-0.476)--(13.11-1-0.946-0.789-0.548-0.245-0.082-0.402-0.677-0.880-0.486,0+0.324+0.614+0.837+0.969-0.997-0.916-0.736-0.476-0.123);
\end{tikzpicture}
\caption{The open sets $V,S_0(V),\cdots,S_m(V)$ and geometrical relation for $m\equiv3$ mod $4$ where $m\ge4$.}\label{3}
\end{figure}

\begin{ack}
The author thanks Professor Jean-Paul Allouche for his suggestions, and thanks the Oversea Study Program of Guangzhou Elite Project (GEP) for financial support (JY201815).
\end{ack}

\end{document}